\documentclass[reqno]{amsart}

\usepackage{amsmath}
\usepackage{amsfonts}
\usepackage{amsthm}
\usepackage{amssymb}
\usepackage{graphicx}
\usepackage{graphics}
\usepackage{bm}
\usepackage{dsfont}
\usepackage{color}
\usepackage{enumerate}
\usepackage[font=footnotesize]{caption}
\usepackage[all]{xy}
\numberwithin{equation}{section}
\newtheoremstyle{personal}%
{12pt}%      Space above
{12pt}%      Space below
{\slshape}%         Body font
{}%         Indent amount
{\bfseries}% Theorem head font
{.}%        Punctuation after theorem head
{.5em}%     Space after theorem head
{}%         Theorem head spec (can be left empty, meaning "normal")
\theoremstyle{personal}%
\newtheorem{thm}{Theorem}[section]
\newtheorem{cor}[thm]{Corollary}
\newtheorem{lem}[thm]{Lemma}
\newtheorem{prop}[thm]{Proposition}
\theoremstyle{definition}
\newtheorem{rem}[thm]{Remark}

\newtheorem{dfn}[thm]{Definition}
\theoremstyle{remark}
\newtheorem{rmk}[thm]{Remark}
\newcommand{\orb}{\mathrm{orb}}
\newcommand{\N}{\mathds{N}}
\newcommand{\Z}{\mathds{Z}}
\newcommand{\Q}{\mathbb{Q}}
\newcommand{\R}{\mathds{R}}
\renewcommand{\P}{\mathcal{P}}
\newcommand{\E}{\mathbb{E}}
\newcommand{\OO}{\mathcal{O}}

\newcommand{\T}{\mathds{T}}
\newcommand{\K}{\mathcal{K}}

\newcommand{\M}{\mathcal{M}}
\newcommand{\U}{\mathcal{U}}

\newcommand{\G}{\mathcal{G}}
\renewcommand{\O}{\mathrm{O}}

\newcommand{\V}{\mathcal{V}}

\newcommand{\ps}{\partial_s}
\newcommand{\D}{\mathcal{D}}
\newcommand{\e}{\varepsilon}
\newcommand{\<}{\langle}
\renewcommand{\>}{\rangle}
\newcommand{\na}{\nabla}
\newcommand{\Sb}{\mathbb{S}}
\newcommand{\A}{\mathbb{A}}
\newcommand{\pt}{\partial_t}

\newcommand{\vp}{\varphi}
\newcommand{\dist}{\mathrm{dist}}

\newcommand{\z}{\mathfrak{z}}
\DeclareMathOperator{\ad}{ad}
\DeclareMathOperator{\Ad}{Ad}
\DeclareMathOperator{\eu}{eu}

\DeclareMathOperator{\grad}{grad}

\newcommand{\g}{\mathfrak{g}}
\renewcommand{\o}{\mathfrak{o}}
\renewcommand{\t}{\mathfrak{t}}
\newcommand{\ul}[1]{\underline{#1}}
\newcommand{\ol}[1]{\overline{#1}}

\newcommand{\hor}{\mathrm{hor}}

\begin{document}

\title[$G$-geodesics and closed magnetic geodesics on orbifolds]{On geodesic flows with symmetries and closed magnetic geodesics on orbifolds}

\author[L. Asselle]{Luca Asselle}
\address{Ruhr Universit\"at Bochum, Fakult\"at f\"ur Mathematik, Universit\"atsstra\ss e 150 \newline\indent Geb\"aude NA 4/35, D-44801 Bochum, Germany}
\email{luca.asselle@ruhr-uni-bochum.de}

\author[F. Schm\"aschke]{Felix Schm\"aschke}
\address{Humboldt-Universit\"at zu Berlin, Unter den Linden 6 \newline\indent 10099, Berlin, Germany}
\email{felix.schmaeschke@math.hu-berlin.de}

\date{November 21, 2017}
\subjclass[2000]{37J15, 37J45, 53C35, 53D20, 58E05.}
\keywords{Magnetic flows, periodic orbits, Ma\~n\'e critical values, Rabinowitz action functional, Symplectic reduction.}

\begin{abstract}
Let $Q$ be a closed manifold admitting a locally-free action of a compact Lie group $G$. In this paper we study the properties of geodesic flows on $Q$ given by Riemannian metrics which are invariant by such an action. In particular, we will be interested in the existence of geodesics which are closed up to the action of some element in the group $G$, since they project to closed magnetic geodesics on the quotient orbifold $Q/G$.
\tableofcontents
\end{abstract}

\maketitle
\vspace{-10mm}
\section{Introduction}
\label{s:introduction}

Let $Q$ be a \emph{closed locally-free principal $G$-bundle}, that is a smooth closed manifold equipped with a smooth, effective and locally-free action of a compact Lie group $G$. On $Q$  we consider a Riemannian metric $g_Q$, which is $G$-invariant and restricts to an $\Ad$-invariant bilinear form on the subspace of fundamental vector fields. The study of the geodesic flow of such a Riemannian manifold 
$(Q,g_Q)$ is made particularly interesting by the fact that from the existence of geodesics which are closed up to $G$-action and satisfy some additional constraint (roughly speaking, the angle that such geodesics form with fundamental vector fields is in any point constant and equal to some a priori fixed angle), \emph{constrained $G$-closed geodesics} for short, one obtains the existence of closed magnetic geodesics with given energy 
on the quotient orbifold $Q/G$. This approach has been already implemented in \cite{Asselle:2017} in the particular case of free principal $S^1$-bundle and yielded an alternative 
proof of the existence of one closed magnetic geodesic for almost evergy energy level on every closed non-aspherical Riemannian manifold equipped with a closed 
two-form representing an integer cohomology class.

More precisely, let $(M,g_M)$ be a closed (throughout this paper always effective) Riemannian orbifold and $\sigma$ be a closed two-form on $M$. 
A \textit{closed magnetic geodesic for the pair} $(g_M,\sigma)$ is a loop $\mu:[0,T]\rightarrow M$ which locally lifts to a classical magnetic geodesic, meaning that for every $t\in [0,T]$ there exists $\epsilon >0$ small enough such that the 
restriction of $\mu$ to $(t-\epsilon,t+\epsilon)$ is entirely contained in an orbifold chart $(U,\Gamma,\varphi)$ and any lift $\tilde \mu:(t-\epsilon,t+\epsilon)\rightarrow U$ of $\mu$ to $U$ is a magnetic geodesic (in the classical sense) for the Riemannian metric $\tilde g_M$ and the two-form $\tilde \sigma$ obtained respectively by lifting the Riemannian metric $g_M$ and the two-form $\sigma$ to $U$. We say that the closed magnetic geodesic $\mu$ has \textit{energy} $k$, if every local lift $\tilde \mu$ has energy $k$ in the classical sense.  We readily see that this definition naturally extends the usual definition of closed geodesics for Riemannian orbifolds (see e.g.\ \cite{Haefliger:2006} or \cite{Lange:2016}). 
As far was we know, magnetic flows on Riemannian orbifolds - in contrast with geodesic flows (see e.g.\ \cite{Haefliger:2006,Lange:2016}) - have not been studied yet. Nevertheless, the corresponding problem for manifolds has received the attention of numerous great mathematicians over the past few decades (e.g.\ Contreras,
Ginzburg, Taimanov, just to mention some of them) and a rich literature on the topic is nowadays available (see e.g.\  \cite{Abbondandolo:2014rb,Asselle:2016qv,Benedetti:2014,Benedetti:2015,Contreras:2006yo,Ginzburg:1994,Ginzburg:2009,Schneider:2011yw,Schneider:2012ny,Schneider:2012xm,Taimanov:1991el,Taimanov:1992fs}). For generalities about magnetic flows on manifolds we refer to \cite{Asselle:2014hc} and references therein.

Closed magnetic geodesics on orbifolds are related with constrained $G$-closed geodesics by - roughly speaking - projection, meaning that every closed magnetic geodesic on $M$ for the pair $(g_M,\sigma)$ lifts to a constrained $G$-closed geodesic on a suitable Riemannian locally free $G$ principal bundle $(Q,g_Q)$ over $M$ and, conversely, every constrained $G$-closed geodesic on $(Q,g_Q)$ projects to a closed magnetic geodesic on $M$ for the pair $(g_M,\sigma)$. The latter (and hence also the former via projection) turn out correspond to critical points of the functional 
$\Sb_k:W^{1,2}(S^1,Q)\times L^2(S^1,\g)\times (0,+\infty)\to \R$
given by
\begin{equation}
\Sb_k(\gamma,\phi,T) = \frac{1}{2T}\int_0^1 |\dot \gamma(t) + \ul{\phi(t)}(\gamma(t))|^2 \, dt + \int_0^1 \<\phi(t),Z\>\, dt + kT.
\label{functionalsk}
\end{equation}
Here $\g$ denotes the Lie algebra of the compact Lie group $G$ acting on $Q$, $Z \in \g$ is a suitable central vector of unit length with respect to some fixed $\Ad$-invariant metric on $\g$, and $\ul{\phi(t)}$ denotes the fundamental vector field on $Q$ associated with the Lie algebra element $\phi(t)$.
All definitions will be given rigorously in Sections \ref{s:orbifolds} and \ref{s:locallyfreeactions}; the relation between constrained $G$-closed geodesics and critical points of $\Sb_k$ will be explained in detail in Section \ref{s:thefunctionalsk}.

The main result of the present paper is the following generalization of the main theorem in \cite{Asselle:2014hc}. In the statement below $\pi_\ell^\orb(M)$ denote 
the orbifold-theoretic homotopy groups as defined in \cite[Def.\ 1.50]{Orbifolds}.

\begin{thm}\label{thm:main}
Let $(M,g_M)$ be a closed non-rationally aspherical Riemannian orbifold, i.e. such that $\pi^\orb_\ell(M) \otimes \Q \neq 0$ for some $\ell\geq 2$, and $\sigma$ be a closed $2$-form on $M$. Then for almost every $k>0$ there exists a contractible closed magnetic geodesic for the pair $(g_M,\sigma)$ with energy $k$.
\end{thm}

\begin{cor}\label{cor:main}
Suppose that $(M,g_M)$ is a closed Riemannian orbifold and $\sigma$ is a closed $2$-form on $M$ such that one of the following conditions is satisfied:
\begin{enumerate}[i)]
\item $\pi_1^{\text{orb}}(M)$ is finite.
\item $\sigma$ is not weakly-exact (i.e. its lift to any cover is not exact).
\end{enumerate}
 Then for almost every $k>0$ there exists a closed magnetic geodesic for the pair $(g_M,\sigma)$ with energy $k$. 
\end{cor}

\vspace{-1mm}

An orbifold $M$ is said \emph{developable} if it is isomorphic to a quotient $\tilde M/\Lambda$ where $\tilde M$ is a manifold and $\Lambda$ is a discrete (not necessarily finite) group acting properly on $\tilde M$. Clearly, not all orbifolds (e.g.\ weighted projective spaces, see \cite[Example 1.15]{Orbifolds}) are of this type. Such orbifolds are called  \emph{non-developable}. In this case, adapting the argument in \cite{Haefliger:2006} to our setting yields the following 

\vspace{-1mm}

\begin{thm}\label{thm:main3}
Let $(M,g_M)$ be a non-developable Riemmanian orbifold and $\sigma \in \Omega^2(M)$ be a closed two form on $M$. Then for almost every $k >0$ there exists a contractible closed magnetic geodesic for the pair $(g_M,\sigma)$ with energy $k$. 
\end{thm}

\vspace{-1mm}

The method we will use to prove Theorems \ref{thm:main}, \ref{thm:main3} and Corollary \ref{cor:main} is based on critical point theory for the functional $\Sb_k$ and consists 
roughly speaking on building, under the given assumptions, a non-trivial minimax class for $\Sb_k$ to which (an infinite dimensional version of) Morse theory will be applied. 

\begin{rem} 
We claim Theorem \ref{thm:main} to be true also if $M$ is merely \textit{non-aspherical}, i.e. $\pi_\ell^{\text{orb}}(M)\neq 0$ for some $\ell \geq 2$. However, in this case we are not able  to prove an analogue of Lemma \ref{lem:enonaspherical}, which is the key tool to build a non-trivial minimax class for $\Sb_k$. If this could be fixed, then we would automatically obtain the existence of a contractible closed magnetic geodesic for almost every energy, provided 
the orbifold is non-developable or non-aspherical. 
We shall also notice that this is the best result one can hope for; indeed, developable aspherical orbifolds might not have contractible closed magnetic geodesics at all, since this is already the case for aspherical manifolds (at least if the energy is ``large enough'', see e.g.\ \cite{Abbondandolo:2013is}).
\end{rem}

With the same approach we can also study the existence of closed geodesics on Riemannian orbifolds, for they correspond to critical points of the functional 
$$\E :W^{1,2}(S^1,Q)\times L^2(S^1,\g)\to \R,\quad \E (\gamma,\phi):= \int_0^1 |\dot \gamma + \ul \phi (\gamma)|^2 \, dt.$$

Going along the line of the proof of Theorem \ref{thm:main}, in Section \ref{s:proofoftheorem1.5} we give an alternative proof of Theorem 5.1.1 in \cite{Haefliger:2006} 
about the existence of non-constant closed geodesics on closed orbifolds. 

 \begin{thm}\label{thm:main2}
A closed Riemannian orbifold $(M,g_M)$ carries a closed geodesic, provided one of the following conditions is satisfied:
\begin{enumerate}
\item $M$ is not developable.
\item $\pi_1^{\text{orb}}(M)$ is either finite or contains an element of infinite order.
\end{enumerate}
\end{thm}

\vspace{-2mm}

To the author's best knowledge, all known methods to produce closed geodesics with positive length fail for general developable orbifolds whose orbifold-theoretic fundamental group  is a so called \textit{monster group}, that is an infinite finitely presented group whose elements are all torsion (see \cite{Haefliger:2006}), even though recent developments provided positive answers for suitable subclasses of such orbifolds (see \cite{Dragomir:2014} for more details). 
It is however worth notice that it is not known whether this situation can actually occur or not,
for it is still an open problem to determine whether there are examples of such monster groups\footnote{There are actually evidences that such groups might exist (see e.g. \cite{Yu:2002}). Moreover, many similar examples are known under the assumption for the group to be only finitely generated.}. Nevertheless, the methods developed in this paper could be used to treat such orbifolds and potentially yield new results; this is subject of ongoing research.

\vspace{2mm}

We finish this introduction with a brief summary of the contents of the paper: 

\begin{itemize}
\item In Section \ref{s:orbifolds} we recall definition and basic properties of orbifolds. 
\item In Section \ref{s:locallyfreeactions} we construct the locally-free principal $G$-bundle $Q$ starting from a closed Riemannian orbifold $(M,g_M)$ and a closed two-form $\sigma$ on $M$. 
\item In Section \ref{s:symplecticreduction} we explain the relation between constrained $G$-closed geodesics and critical points of a suitable Rabinowitz-type action functional $\A_k$ defined over the space of loops in $T^*Q \times \g$.
\item In Section \ref{s:thefunctionalsk} we introduce the functional $\Sb_k$ and study its properties. 
\item In Section \ref{s:proofoftheorem1.1} we prove Theorem \ref{thm:main} and Corollary \ref{cor:main}. 
\item Finally, in Section \ref{s:proofoftheorem1.5} we prove Theorem \ref{thm:main2}.
\end{itemize}

\vspace{3mm}

\noindent \textbf{Acknowledgments.} We warmly thank Christian Lange for pointing out a mistake in a previous version of the proof of Theorem 1.5. 
The statement of the theorem has been corrected accordingly. 
L. A. is partially supported by the DFG grants AB 360/2-1 ``Periodic orbits of
conservative systems below the Ma$\tilde{\text{n}}$\'e critical energy
value" and AS 546/1-1 "Morse theoretical methods in Hamiltonian dynamics".

%%%%%%%%%%%%%%%%%%%%%%%%%%%%%%%%%%%%
%%%%%%%%%%%%%%%%%%%%%%%%%%%%%%%%%%%%
%%%%%%%%%%%%%%%%%%%%%%%%%%%%%%%%%%%%
\section{Orbifolds}
\label{s:orbifolds}
Orbifolds appear naturally in several different fields of mathematics, including representation theory, algebraic geometry, physics and topology.  Roughly speaking, they are generalizations of manifolds by allowing certain singularities. As such they have many properties in common with manifolds and share known constructions, such as tangent bundles, differential forms, vector fields etc. We quickly recall these definitions and  highlight basic properties, accounting to \cite{Orbifolds}.

Let $M$ be a paracompact Hausdorff space and fix $n \in \N$. An \emph{orbifold chart} is a triple $(U,\Gamma,\vp)$ such that $U \subset \R^n$ is a connected open subset, $\Gamma$ is a finite group acting effectively on $U$ by smooth automorphisms,  and $\vp:U \to M$ is a $\Gamma$-invariant map inducing a homeomorphism of $U/\Gamma$ onto its image. An \emph{embedding} $(U,\Gamma,\vp) \hookrightarrow  (U',\Gamma',\vp')$ of orbifold charts is an embedding $\lambda:U\hookrightarrow U'$ such that $\vp \circ \lambda = \vp'$. An \emph{orbifold atlas} $\U = \{(U,\Gamma,\vp)\}$ is a set of \textit{compatible} orbifold charts such that $\bigcup_{\U} \vp(U) = M$: for all $p \in M$ and all orbifold charts $(U_1,\Gamma_1,\vp_1),(U_2,\Gamma_2,\vp_2) \in \U$ with $p \in \vp_1(U_1) \cap \vp_2(U_2)$ there exists $(U,\Gamma,\vp) \in \U$ such that $p \in \vp(U)$ and embeddings $(U,\Gamma,\vp) \hookrightarrow (U_i,\Gamma_i,\vp_i)$ for $i=1,2$. We define the \emph{transition function} near $p$ as the diffeomorphism $\phi_{12}=\lambda_2\circ \lambda_1^{-1}:\lambda_1(U) \to \lambda_2(U)$. We say that an atlas $\U$ is a \emph{refinement} of $\U'$ if for every chart of $\U$ there exists an embedding into some chart of $\U'$. Two atlases are \emph{equivalent} if they have a common refinement. Every orbifold atlas has a unique maximal refinement and two orbifold atlases are equivalent if and only if they have the same maximal refinement. 

\begin{dfn}
An \emph{effective orbifold of dimension $n$} is a paracompact Hausdorff space $M$ equipped with an equivalence class of $n$-dimensional orbifold atlases. 
\end{dfn}

As the definition above suggests, there is a further generalization of the notion of an orbifold by allowing non-effective local actions. However, we will not consider these objects in this paper and refer to effective orbifolds simply as orbifolds. 

\begin{dfn}\label{dfn:orbismooth}
A map $f:M \to N$ between two orbifolds is \emph{smooth}, if for every $p \in M$ there exist charts $(U,\Gamma,\vp)$ and $(U',\Gamma',\vp')$ with $p \in \vp(U)$  and $f(p) \in \vp'(U')$ together with  a smooth map $\tilde f:U \to U'$ satisfying $f \circ \vp = \vp' \circ \tilde f$.
\end{dfn}

Exactly as for manifolds we define the tangent bundle of an orbifold $M$ by gluing together tangent bundles of the local charts using transition functions. 

\begin{dfn}
We define the \emph{tangent bundle} as the space
\[
TM = \left.\bigsqcup_{(U,\Gamma,\vp) \in \U} (U \times_\Gamma \R^n)\right\slash {\sim}\,,
\] 
where $U \times_\Gamma \R^n$ denotes the quotient space of $U \times \R^n$ by the diagonal action of $\Gamma$ and $\sim$ is defined via $[p,v] \sim [p',v']$ if and only if  $\phi_{12}(p)=p'$ and $d_p \phi_{12} (v) =v'$ for all charts $(U,\Gamma,\vp)$, $(U',\Gamma',\vp') \in \U$ such that $p \in U$ and $p'\in U'$.
\end{dfn}

The tangent bundle of an orbifold carries a natural orbifold structure and there is a canonical foot-point projection map $TM \to M$. It is worth mentioning that fibres are no longer vector spaces but rather quotients $\R^n/\Gamma_p$, where the finite group $\Gamma_p$ varies with $p \in M$.
By a similar gluing construction we define the cotangent bundle $T^*M$ and its exterior as well as symmetric tensor powers. Further we define \emph{vector fields}, \emph{differential forms} and \emph{Riemannian metrics} as smooth sections of such bundles. One shows that by virtue of the definitions we find local representatives of such sections which in addition of satisfying the usual transformation rules are equivariant with respect to the local group action. In particular integral curves to a local representative of  a vector field depend equivariantly on the starting point, which implies that vector fields on orbifolds induce flows as usual. We define metric connections verbatim to the manifold definition and by the previous observation we see that the geodesic flow and the magnetic geodesic flow on an orbifold are well-defined. In this sense, a magnetic geodesic on $M$ is a curve which locally (i.e. in every orbifold chart) lifts to a classical magnetic geodesic. More precisely, given a Riemannian metric $g_M$ and a closed 2-form $\sigma$ on $M$, a path 
$$\mu:(a,b) \to M$$ is said to be a \textit{magnetic geodesic} if for each $t_0 \in (a,b)$ there exists $\e>0$, an orbifold chart $(U,\Gamma,\vp)$ such that $\mu(t) \in \vp(U)$ for all $t\in (t_0-\e,t_0+t)$, and a smooth map $\tilde \mu:(t_0-\e,t_0 + \e) \to U$ with $\mu|_{(t_0-\e,t_0+\e)}= \vp \circ \tilde \mu$, such that 
\begin{equation}\label{eq:magnetic}
\na_t \dot{\tilde \mu} =  Y_{\tilde \mu}(\dot{\tilde \mu}).
\end{equation}
Here $\na$ and $Y$ are respectively the Levi-Civita connection and the Lorenz force associated with the local representatives of $g_M$ and $\sigma$. 
We say that the magnetic geodesic $\mu$ has \textit{energy} $k$, if every lift $\tilde \mu$ has energy $k$ in the classical sense.

In the statement of the next theorem $\Omega^*(M)$ denotes the deRahm cochain complex of $M$.

\begin{thm}[Satake, {\cite[Page 34]{Orbifolds}}]\label{thm:Satake}
 The cohomology of $(\Omega^*(M),d)$ is canonically isomorphic to $H^*(M;\R)$.
\end{thm}

An important source of examples of orbifolds are quotients of smooth manifolds by compact Lie groups which act \emph{locally-freely}, i.e. all stabilizer groups are finite. Orbifold charts for such quotients are provided by the slice theorem. As it turns out, every orbifold is of this form. We cite~\cite[Corollary 1.24]{Orbifolds}.

\begin{prop}\label{prop:quotient_orbifold}
Every $n$-orbifold is diffeomorphic to a quotient orbifold for a smooth, effective and locally-free $\O(n)$-action on a smooth manifold. 
\end{prop}

%%%%%%%%%%%%%%%%%%%%%%%%%%%%%%%%%%%%%%%
\section{Locally free actions}
\label{s:locallyfreeactions}
Let $Q$ be a smooth manifold equipped with a smooth action of a compact Lie group $G$, denoted $G \times Q \to Q$, $(g,q) \mapsto g\cdot q$.  Throughout the paper we will assume the action to be \textit{effective} and \textit{locally free}, i.e. with all stabilizer groups finite. Let $\g$ denote the Lie-algebra of $G$ and $\Ad_g:\g \to \g$ the adjoint action map for $g \in G$. For any $X \in \g$ we denote by $\ul X:Q \to TQ$ the \emph{fundamental vector field} on $Q$ defined  by 
\begin{equation}\label{eq:fundamental_vectorfield}
\ul X_{q} = \frac{d}{dt}\Big |_{t=0} \exp(t X)\cdot q  \in T_q Q, \quad \forall q\in Q.  
\end{equation}
Observe that the map $\g \to T_qQ$, $X \mapsto \ul X_q$ is injective for all $q \in Q$.
Since locally there is no difference between locally free and free actions, the definition of a principal connection form literally carries over.
\begin{dfn}
A \emph{principal connection form} for $Q$ is a $\g$-valued differential one-form $\theta \in \Omega^1(Q;\g)$ such that 
\begin{enumerate}[(i)]
\item $\theta(\ul X) = X$ for all $X \in \g$,
\item $\vp_g ^* \theta = \Ad_g  \theta$ for all $g \in G$, where $\vp_g:Q\to Q$, $q\mapsto g\cdot q$.
\end{enumerate}
\end{dfn}

\begin{lem}
The set of principal connections for $Q$ is a non-empty affine space.
\end{lem}
\vspace{-4mm}
\begin{proof}
The fact that the set of principal connections is an affine space follows directly from the definition. 
Consider now any metric $g_Q$ on $Q$. After averaging we assume that $g_Q$ is $G$-invariant. Given a point $q \in Q$, denote with $G\cdot q \subset Q$ its orbit under the $G$-action. We define the connection form $\theta$ at $q$ as the $g_Q$-orthogonal projection $T_q Q \to T_q (G\cdot q)$ composed with the inverse of the canonical isomorphism $\g \cong T_q(G\cdot q)$. By construction $\theta$ satisfies the first property. The second property follows from the identity
\[
(d \vp_g\ul X)_{gq} = \frac{d}{dt}\Big|_{t=0} g\exp(t X)q = \frac{d}{dt}\Big|_{t=0} g\exp(t X)g^{-1}gq = \ul{\Ad_g X}_{gq}\,,
\]
for all $q\in Q$ and $g\in G$ and the fact that $d\vp_g$ is an isometry.
\end{proof}
\begin{dfn}\label{dfn:sigma}
Given a principal connection form $\theta$ we define its \emph{curvature form} by
\[\sigma = d\theta+[\theta,\theta]\in \Omega^2(Q;\g).\]
More precisely, for any $Z \in \g^\vee$, $p \in Q$ and $u,v \in T_pQ$ we define
\begin{equation}\label{eq:sigmaX}
\sigma_Z(u,v)= d\theta_Z(u,v)+\<Z,[\theta(u),\theta(v)]\>\,,
\end{equation}
where $\sigma_Z=\<Z,\sigma\>$, $\theta_Z=\<Z,\theta\>$ and $\<\cdot,\cdot\>$ denotes the duality pairing. 
\end{dfn}

\subsection{Horizontally lifted metrics}\label{sec:mertic} For further computations, we give another description of the curvature form $\sigma$ in terms of a particular metric on $Q$ obtained from a given metric on the quotient and a connection form.  To this end fix an $\Ad$-invariant positive bilinear form on $\g$, which is possible as we assume $G$ to be compact. As already mentioned in Section 2.1, the quotient space $M:=Q/G$ carries a canonical orbifold structure. We denote by $\tau:Q \to M$ the quotient map, which is smooth in the sense of Definition~\ref{dfn:orbismooth}. Assume that $M$ is equipped with a metric $g_M$ and $Q$ is equipped with a connection form $\theta$. The connection form induces a splitting of the tangent bundle of $Q$ into \emph{horizontal} and \emph{vertical} vectors
\begin{equation}\label{eq:splitting}
T_p Q = \ker \theta_p \oplus T_p(G\cdot p).
\end{equation}
For any $p \in Q$ let $\tilde \tau:U \to U'$ be the local representative of $\tau$ with respect to the orbifold charts $(U,0,\vp)$ and $(U',\Gamma',\vp')$ around $p$ and $\tau(p)$ respectively and observe that the metric $g_M$ induces a metric on $U'$ which is $\Gamma'$-invariant. We now define the metric $g_Q$ on $Q$ at $p$ by requiring that:
\begin{itemize}
\item the splitting~\eqref{eq:splitting} is orthogonal,
\item the isomorphism $d_p \tilde \tau:\ker \theta_p \to T_{\tilde \tau(p)} U'$ is an isometry,
\item the isomorphism $\g \to T_p(G \cdot p)$, $X \mapsto \ul X$ is an isometry.
\end{itemize}

It is easy to see that the metric $g_Q$ is smooth, $G$-invariant and well-defined regardless of the choice of the charts. In the following we will denote the metrics $g_Q$, $g_M$ and the bilinear form on $\g$ all with $\<\cdot,\cdot\>$ if they can not be confused. Furthermore, we identify $\g^\vee$ with $\g$ using the bilinear form. In what follows, $\na$ denote the Levi-Civita connection for $g_Q$.
\begin{lem}\label{lem:antisymmetric}
For every $X \in \g$ and all $p \in Q$ the bilinear form 
$$T_pQ \times T_pQ \ni (u,v) \mapsto \<u,\na_v \ul X\>$$
 is anti-symmetric.
\end{lem}
\begin{proof}
Let $\Phi_{\ul X}^t$ denote the flow of $\ul X$. Extend $u,v$ via $u(t) = d\Phi_{\ul X}^t u$ and $v(t) = d\Phi_{\ul X}^t(v)$. In particular, $[\ul X,u]=[\ul X,v]=0$.  Since $d\Phi_{\ul X}^t$ is an isometry we have $\ul X \<u,v\>= 0$ and hence  $\<\na_{\ul X} u,v\> = -\<u,\na_{\ul X} v\>$. It follows that 
$$\hspace{22mm} \<u,\na_v \ul X\> = \<u,\na_{\ul X} v\> = - \<\na_{\ul X} u,v\> = -\<\na_u \ul X,v\>. \hspace{23mm} \qedhere$$
\end{proof}
\begin{lem}\label{lem:sigma}
We have 
\begin{equation}
\label{eq:sigmana}
\sigma_X(u,v)  = -2\<u,\na_v \ul X\>\,,
\end{equation}
for all $X \in \g$, $p \in Q$ and horizontal vectors $u,v \in \ker \theta_p$. Moreover both sides vanish if one of $u,v$ is horizontal and the other vertical. 
\end{lem} 
\begin{proof}
Extend $u$ and $v$ to vector fields, which we still denoted by $u$ and $v$ respectively. If $u$ and $v$ are horizontal, then by~\eqref{eq:sigmaX} we have
\begin{align*}
\sigma_X(u,v) &=  u(\theta_X(v)) - v(\theta_X(u)) - \theta_X([u,v])\,.\\
&= - \<[u,v],\ul X\>\\
&= -\<\na_u v,\ul X\> +\<\na_v u,\ul X\>\\
&=\<v,\na_u \ul X\> - \<u, \na_v \ul X\>\\
&=-2\<u,\na_v \ul X\>\,.
\end{align*}
where in the last step used the anti-symmetry and in the second-last the identity
$$0=u\<v,\ul X\> = \<\na_u v,\ul X\> + \<v,\na_u \ul X\>.$$ 

Assume now that $u$ is horizontal and $v$ is vertical. We also assume without loss of generality that $v=\ul V$ is the fundamental vector field associated with $V \in \g$. Again by~\eqref{eq:sigmaX} we have
\begin{align*}
\sigma_X(u,v)&= d\theta_X(u,v) = u(\theta_X(v)) - v(\theta_X(u)) - \theta_X([u,v])\,.\\
&=u\<\ul X,\ul V\>-\<\ul X,[u,\ul V]\>=0,
\end{align*}
as $\<\ul X,\ul V\>=\<X,V\>$ is constant and $[u,\ul V]=0$, for horizontal vector fields are invariant under the flow of $\ul V$. In remains to check that $\<u,\na_v X\>$ vanishes. Indeed, since $[\ul X,\ul V]=\ul{[X,V]}$ is vertical, we have $\<u,\na_{\ul X} \ul V\> = \<u,\na_{\ul V} \ul {X}\>$. By anti-symmetry on both sides we have $\<\ul X,\na_u \ul V\> = \<\ul V,\na_u \ul X\>$. Since $\<\ul X,\ul V\>=\<X,V\>$ is constant, we have $0=u\<\ul X,\ul V\>= \<\na_u \ul X,\ul V\> + \<\ul X,\na_u \ul V\>$. This shows $\<\ul V,\na_u \ul X\> = -\<\ul X,\na_u \ul V\>=-\<\ul V,\na_u \ul X\>$. Thus $\<u,\na_{\ul V} \ul X\> =0$ as required. 
\end{proof}
\begin{rem}
Equation~\eqref{eq:sigmana} is not true for all $u,v$. In fact, if $u$ and $v$ are both vertical and respectively given by the fundamental vector fields $u=\ul U$ and $v=\ul V$ with $U,V \in \g$, then we have $\sigma_X(u,v) = 0$ but $\<u,\na_v \ul X\> = \<[\ul U,\ul V],\ul X\>$. 
\end{rem}
%%%%%%%%%%%%%%%%%%%%%
\subsection{Constructing the bundle.}
\label{constructingthebundle}
We now show that for any closed orbifold $M$ together with a closed two-form
$\sigma$, we can find a locally-free principal $G$-bundle $Q$, $G$ compact Lie group, such that $Q/G$ and $M$ are isomorphic as orbifolds and 
$\sigma$ arises as a curvature form. Recall that $Z \in \g^\vee$ is called \emph{central} if $\Ad^*_g Z =Z$ for all $g \in G$. 
\begin{prop}\label{prop:orb}
  Let $M$ be a closed orbifold and $\sigma \in \Omega^2(M)$ be a closed 2-form. Then there exist:
\begin{itemize}
\item a locally-free principal $G$-bundle $Q$, $G$ compact Lie group,
\item a connection form $\theta \in \Omega^1(Q,\g)$, 
\item a central covector $Z \in \g^\vee$ with $Z\neq 0$,
\end{itemize}  
such that $M\cong Q/G$ as orbifolds and $\tau^*\sigma=\sigma_Z$, where $\sigma_Z$ is defined as in \eqref{eq:sigmaX}.
\end{prop}

Before proving the proposition we recall the definition of the Euler class for
(orientable) circle bundles or, equivalently, principal $S^1$-bundles,
where $S^1=\R/\Z$ is viewed as a Lie group. We give the definition of the
Euler class using \v{C}ech cohomology (see \cite[pp. 35]{GriffithsHarris} for generalities about sheaves
and \v{C}ech cohomology). Let $P$ be a smooth manifold,
$\rho:Q \to P$ be a principal $S^1$-bundle. Choose an open cover
$\U=(U_i \subset P)_{i \in I}$ of $P$ such that $Q|_{U_i}$ is trivial
for all $i \in I$. Let $g_{ij}:U_i \cap U_j \to S^1$ be the
corresponding trivialization change maps for all $i,j \in I$. The
family $g=(g_{ij})$ is a \v{C}ech-cocycle and defines a
\v{C}ech-cohomology class $[g] \in \check{H}^1(P,\OO_{S^1})$, where by
$\OO_{S^1}$ we denote the sheaf of smooth maps with values in
$S^1$. The short exact sequence of groups $0\to\Z \to \R \to S^1\to 0$
induces a long exact sequence for the \v{C}ech-cohomology groups
\begin{equation}
  \label{eq:Cechsequence}
\dots \to \check{H}^1(P;\OO_\R) \to \check{H}^1(P;\OO_{S^1})
\stackrel{\delta}{\to} \check{H}^2(P;\OO_\Z) \to
\check{H}^2(P;\OO_\R)\to \dots\,,  
\end{equation}
where by $\OO_\R$, $\OO_\Z$ we denote the sheaf of smooth maps with
values in $\R$, $\Z$ respectively. Note that $\OO_\Z$ is also the
sheaf of locally constant maps  denoted $\ul \Z$. There is
canonical isomorphism $\check{H}^2(P;\ul\Z) \cong H^2(P;\Z)$ where the
right-hand side denotes the singular cohomology with integer
coefficients (cf.\ \cite[Theorem III.1.1]{Bredon:Sheaf}). We identify these
groups without further mentioning and define the \emph{Euler class}
\[\eu(Q):=\delta([g]) \in H^2(P;\Z)\,.\]
The Lie algebra of $S^1$ is canonically identified with $\R$ and a
connection form on $Q$ is simply a one form $\theta \in \Omega^1(Q)$
which is $S^1$-invariant and satisfies $\theta(\ul Z)= 1$ where $\ul
Z$ is the fundamental vector field of the fibre rotations. By that
requirement we have $d\theta(\ul Z,\cdot) =0$ and hence the
\emph{curvature form} $\sigma\in \Omega^2(P)$ is well-defined via
$\rho^*\sigma =d\theta$. By construction we have $d\sigma=0$. We
identify the deRahm cohomology with $H^2(P;\R)$ and denote by
$[\sigma] \in H^2(P;\R)$ the  cohomology class represented by
$\sigma$. Further we denote by $\eu(Q)_{\R} \in H^2(P;\R)$ the element
$\eu(Q) \otimes 1$ under the canonical isomorphism $H^2(P;\Z)\otimes
\R \cong H^2(P;\R)$.

\begin{lem}\label{lem:eulerclass_equals_sigma}
  We have $\eu(Q)_{\R} = [\sigma] \in H^2(P;\R)$.
\end{lem}

\begin{proof}
  See \cite[Page 141]{GriffithsHarris} for the corresponding proof for
  complex line bundles. Assume without loss of generality that $\U$ is
  a good cover (i.e. finite intersections $U_{i_1}\cap U_{i_2}
  \cap \dots \cap U_{i_k}$ are contractible for all
  $(i_1,i_2,\dots,i_k) \subset I$ and $k \geq 1$) so that $\check{H}^*(\U;\OO_{S^1})
  \cong \check{H}^*(P;\OO_{S^1})$. Let $f_{ij}:U_i \cap U_j \to \R$ be
  lifts of $g_{ij}$ for all $i,j \in I$. The tuple $f=(f_{ij}) \in
  \check{C}^1(\U;\OO_\R)$ is a \v{C}ech-cochain with differential
  \[
  w=(w_{ijk})=\check{d}f, \qquad w_{ijk} = f_{ij} - f_{ik} +f_{jk}\,.
  \]
  Since $(g_{ij})$ is a cocycle we must have $w_{ijk} \in \Z$ for all
  $i,j,k\in I$. Thus $w \in \check{C}^2(\U;\ul \Z)$ is a cocycle,
  which by definition of the connecting homomorphism $\delta$
  represents the cohomology class of $\eu(Q)=[w]$.
  
  To compare $[w]$ with $[\sigma]$ we need to examine the isomorphism
  between deRahm cohomology and \v{C}ech cohomology as given in
  \cite[Section 8]{BottTu} or \cite[Page 43]{GriffithsHarris}.  In
  short, the cocycle representing the class corresponding to
  $[\sigma]$ under the isomorphism is constructed as follows: First
  find one-forms $\eta_i \in\Omega^1(U_i)$ such that
  $d\eta_i=\sigma|_{U_i}$ for all $i\in I$. Second, find functions
  $h_{ij} \in \Omega^0(U_i\cap U_j)$ such that
  $dh_{ij} = \eta_i -\eta_j|_{U_i\cap U_j}$. Then $u_{ijk} = h_{ij} -
  h_{ik} + h_{jk}|_{U_i \cap U_j \cap U_k}$ is constant for all $i,j,k
  \in I$ and the tuple $u=(u_{ijk}) \in \check{C}^2(\U,\ul\R)$
  represents the image of $[\sigma]$, where $\ul \R$ denotes
  the sheaf of locally constant functions with values in $\R$.
  
  On the other hand if $\vp_i:U_i \times S^1 \to Q|_{U_i}$ denotes the
  trivializations we have that $\vp_i^* \theta = \tilde \eta_i + dt$
  for some $\tilde \eta_i \in \Omega^1(U_i)$. Thus $\sigma|_{U_i} =
  d\tilde \eta_i$ and we assume without loss of generality that
  $\tilde \eta_i = \eta_i$. The trivialization change maps are defined
  via $\vp_i(p,[t]) = \vp_j(p,g_{ij}(p)+[t])$ for all $[t] \in S^1$
  and $p \in U_i \cap U_j$. We conclude that $\eta_i=\eta_j + dg_{ij}$
  or equivalently $\eta_i-\eta_j = dg_{ij}$. Thus, without loss of
  generality $h_{ij} = f_{ij}$ and hence $u_{ijk} = w_{ijk}$.
\end{proof}

\begin{cor}
  Given any $[\sigma] \in H^2(P;\R)$ representing an integer
  cohomology class, there is a circle bundle $Q \to P$ such that
  $\eu(Q)_{\R} = [\sigma]$.
\end{cor}
\begin{proof}
  By definition $[\sigma]$ represents an integer cohomology class, if
  there exists $e \in H^2(P;\Z)$ such that $e_\R=[\sigma]$. The
  existence of partition of unity implies that $\check{H}^p(P;\OO_\R)$
  vanishes for all $p\geq 1$ (cf.\ \cite[Page 42]{GriffithsHarris}).
  Hence by exactness of~\eqref{eq:Cechsequence} the map $\delta$ is an
  isomorphism. In particular we find $[(g_{ij})]\in
  \check{H}^1(P;\OO_{S^1})$ such that $\delta [(g_{ij})]=e$. We obtain
  the corresponding line bundle with trivialization change maps
  $(g_{ij})$ by the gluing construction, i.e. as the space
  \begin{equation}
    \label{eq:gluingE}
  Q  :=  \coprod_{i} U_i \times S^1 /{\sim}    
  \end{equation}
  where $(p,[t]) \sim (p',[t'])$ if and only if 
$$\quad \quad \quad \ \ \ \ \ \ (p,[t]) \in U_i\times S^1,\ (p',[t']) \in U_j\times S^1, \ p=p', \ [t]=[t']+ g_{ij}(p).\quad \quad\ \ \ \qedhere$$ 
\end{proof}

\vspace{2mm}

Assume now that there is a locally-free $G$-action $\vp: G\times P \to P$, $\vp(g,\cdot)=\vp_g$, on $P$. Then, the quotient $M:=P/G$ is an orbifold and we denote by $\tau:P \to M$ the quotient map.

\begin{lem}\label{lem:circlebundle_orbifold}
Given $e \in H^2(M;\Z)$, there exists a principal circle bundle $Q \to P$ equipped with a $G$-action $\tilde \vp:G\times Q \to Q$, $\tilde \vp(g,\cdot)=\tilde \vp_g$ such that the $G$-action commutes with the $S^1$-action, $\tau \circ \tilde \vp_g = \vp_g \circ \tau$ for all $g \in G$, and $\eu(Q) = \tau^*e$. 
\end{lem}
\begin{proof}
  It suffices to find an open cover $\U=(U_i)_{i \in I}$ of $P$ by
  $G$-invariant open subsets $U_i \subset P$ and a cocycle
  $g=(g_{ij}:U_i \cap U_j \to S^1)_{i,j \in I}$ such that
  \begin{itemize}
  \item $g_{ij}$ is $G$-invariant
  \item $\delta([g]) =e$.
  \end{itemize}
  Indeed, the product action of $G \times  S^1$ on the patches $U_i\times S^1$ descends to an action of $G
  \times S^1$ on the quotient $E$ obtained by the gluing construction~\eqref{eq:gluingE}. 

Since $M$ admits a partition of unity, we have $\check{H}^i(M;\OO_\R)=0$ for $i \geq 1$ and replacing $P$ with $M$ in the sequence~\eqref{eq:Cechsequence} we conclude that we have a canonical isomorphism $\bar \delta:\check{H}^1(M;S^1) \cong \check{H}^2(M;\ul \Z)$. Also we a have canonical isomorphism
  $\check{H}^2(M;\ul \Z) \cong H^2(M;\Z)$ (in fact for any paracompact
  Hausdorff space $M$, cf.\ \cite[Theorem III.1.1]{Bredon:Sheaf}) and we identify these groups via the isomorphism. Pick a good cover $\bar U = \{\bar U_i\}_{i \in I}$ of $M$ and consider the pull-back cover $\U=\{U_i\}_{i \in I}$ of $P$ where $U_i = \tau^{-1}(\bar U_i)$. Let $[\bar g] \in \check{H}^1(\bar \U;S^1) \cong \check{H}^1(M;S^1)$ be a class such that $\bar \delta ([\bar g])= e$ and let $g=\tau^*\bar g$ be the pull-back. Recall that this means that $g=(g_{ij})$ is defined by $g_{ij}:U_i \cap U_j \to S^1$, $g_{ij} = \bar g_{ij} \circ \tau$.  Obviously $g_{ij}$ is $G$-invariant and since all involved isomorphism are canonical we also have $\delta( [g])= \tau^* \bar \delta ([\bar g] )= \tau^*e$.
\end{proof}

\begin{proof}[Proof of Proposition~\ref{prop:orb}]
By Proposition~\ref{prop:quotient_orbifold} we know that there exists a smooth manifold $P$ equipped with a locally-free $\O(n)$-action such that $M \cong P/\O(n)$. Let $\tau:P \to M$ be the quotient map. 

Suppose that $\sigma$ is exact and let $\eta$ be such that $d\eta=\sigma$. Then define $Q:= S^1 \times P$ equipped with the product action of the Lie group $G=S^1 \times \O(n)$. Pick any $\Ad$-invariant positive bilinear form on $\g$ such that the splitting $\g= \R \oplus \o(n)$ is orthogonal, where $\o(n)$ denotes the Lie algebra of $\O(n)$, and identify $\g^\vee$ with $\g$ using the bilinear form. Then $Z = (1,0) \in \R \oplus \o(n)$ and the connection form is given by 
$$\theta = \rho^* \theta_0+ (\hat \tau^*\eta + dt)\otimes Z,$$
where $t$ denotes the variable in $S^1$,  $\theta_0$ is any connection for $P$, $\rho:S^1 \times P \to P$ the projection to the second coordinate and $\hat \tau=\tau \circ \rho$. By construction we have $d\<\theta,Z\>=\hat \tau^*d\eta = \hat \tau^*\sigma$. Therefore, we can assume hereafter that $\sigma$ is not exact. 

Let $e_1,\dots,e_m$ be a basis of the free part of $H^2(M;\Z)$ which we identify under the isomorphism $H^2(M;\Z)\otimes \R \cong H^2(M;\R)$ with a basis of $H^2(M;\R)$. Let $Q_1,\dots,Q_m$ be principal circle bundles over $P$ with Euler classes $\tau^*e_1,\dots,\tau^*e_m$ respectively and which are equipped with lifted $\O(n)$-actions in the sense of Lemma~\ref{lem:circlebundle_orbifold}.
Consider the fibre-product
\[
Q =\Big \{ (p_1,\dots,p_m) \in Q_1 \times \dots \times Q_m \ \Big |\ \rho_1(p_1)=\dots=\rho_m(p_m)\Big \},
\]
where $\rho_i:Q_i \to P$ denotes the quotient map. The space $Q$ is a principal $\T^m$-bundle over $P$. Moreover the diagonal $\O(n)$-action on $Q_1\times \dots \times Q_m$ leaves $Q$ invariant and commutes with the $\T^m$-action, so that the corresponding $G:=\T^m \times \O(n)$-action restricted to $Q$ has only finite stabilizers and the quotient is  isomorphic to $M$. 
By Theorem~\ref{thm:Satake} we find closed forms $\sigma_i \in \Omega^2(M)$ such that $[\sigma_i]=e_i$ for $i=1,\dots,m$.  Further we find constants $a_1,\dots,a_m \in \R$ which are not all zero such that $[\sigma] = \sum_{i=1}^m a_i e_i$. Thus, up to changing representatives if necessary, we have 
$\sigma = \sum_{i=1}^m a_i \sigma_i.$ 
By Lemma~\ref{lem:eulerclass_equals_sigma} we can choose connection forms $\theta_i$ on $Q_i$ such that $d\theta_i = \tau_i^* \sigma_i$, where $\tau_i=\tau \circ \rho_i$. We identify the Lie algebra of $\T^m$ with $\R^m$, so that the quotient map is the exponential map. Further we pick any $\Ad$-invariant positive bilinear form on $\g$ such that the splitting $\g=\R^m \oplus \o(n)$ is orthogonal and which is standard, when restricted to $\R^m \oplus 0$. We identify $\g^\vee$ with $\g$ using the bilinear form. Then $Z=(a_1,a_2,\dots,a_m,0)$ and the connection form is 
$$\theta=\rho^*\theta_0 +\sum_{i=1}^m\mathrm{pr}_i^*\theta_i \otimes Z_i\,,$$ 
where $Z_i$ is the $i$-th unit vector if $\R^m$, $\rho:Q \to P$ is the quotient map and $\mathrm{pr}_i:Q \to Q_i$ the projection to the $i$-th factor. By construction we have 
\[d\<\theta,Z\>= \sum_{i=1}^m a_i \mathrm{pr}_i^*d\theta_i =\sum_{i=1}^m a_i \mathrm{pr}_i^* \tau_i^* \sigma_i = \hat \tau^* \sum_{i=1}^m a_i \sigma_i = \hat \tau^*\sigma \,, \] 
where $\hat \tau:Q \to M$ denotes the quotient map. 
\end{proof}

%%%%%%%%%%%%%%%%%%%%%%%%%%%%%%%%
%%%%%%%%%%%%%%%%%%%%%%%%%%%%%%%%
%%%%%%%%%%%%%%%%%%%%%%%%%%%%%%%%

\section{Symplectic reduction}\label{s:symplecticreduction}

As usual we assume that $G$ is a compact Lie group acting smoothly on a manifold $Q$. It is well-known that the
action of $G$ on $Q$ lifts to an Hamiltonian action of $G$ on $T^*Q$
given by $g\cdot (q,p)=(\vp_g(q),\vp_{g^{-1}}^* p)$, where $\vp_g:Q \to Q$ denotes the action by $g \in G$. The corresponding moment map is 
\[
A:T^*Q \to \g^\vee, \qquad (q,p)\mapsto \left( X \mapsto \<p,\underline
  X\>\right)\,,
\]
where $\<\cdot,\cdot\>$ denotes the duality pairing.  Given a central $Z\in \g^\vee$ such that $G$ acts freely on $A^{-1}(Z)$, we consider the \textit{Marsden-Weinstein quotient} by
\[
\big((T^*Q)_Z,\omega_Z\big) \stackrel{\mathrm{def}}{=}
\big(A^{-1}(Z)/G,\omega_\mathrm{red}\big)\,,
\]
in which $\omega_{\mathrm{red}}$ is uniquely determined by $\text{pr}^* \omega_{\mathrm{red}}=\imath^*\omega$, where $\omega$ denotes the standard symplectic form on $T^*Q$ and where $\imath:A^{-1}(Z) \rightarrow T^*Q$ and
$\mathrm{pr}:A^{-1}(Z)\to A^{-1}(Z)/G$ denote the canonical inclusion and the quotient projection respectively. The construction carries over also if $G$ acts merely locally-freely; in this case $(T^*Q)_Z$ is a symplectic orbifold. 

We now recall the well-known fact that $((T^*Q)_Z,\omega_Z)$ is symplectmorphic to a \emph{twisted cotangent bundle}. To this end fix a principal connection $\theta \in \Omega^1(Q,\g)$ and denote with $M=Q/G$ and $\tau:Q \to M$  the base and the quotient map respectively. We define the map $\Pi:A^{-1}(Z)\to T^*M$
implicitly via
\[
\<\Pi(q,p) ,d_q \tau(v)\> = \<p,v\> - \<Z,\theta_q(v)\>\qquad \forall\ v \in T_q Q\,,
\]
where $\<\cdot,\cdot\>$ denotes the duality pairing.
Note that $\Pi$ is well-defined because the kernel of $d_q \tau$
consists precisely of vectors on which the
right-hand side vanishes. It is not hard to see that $\Pi$ is a bundle
map, $G$-invariant and that the fibres are $G$-orbits. We conclude
that $\Pi$ induces a diffeomorphism $(T^*Q)_Z\cong T^*M$.
 
\begin{lem}\label{lmm:magcotangent}
 The map $\Pi$ induces a symplectomorphism 
 \[((T^*Q)_Z,\omega_Z))\cong (T^*M,\bar \omega+\bar
 \pi^*\bar\sigma_Z)\,,\] where $\bar \omega$ is the standard symplectic
 form on $T^*M$, $\bar \pi:T^*M \to M$ the canonical projection and $\bar \sigma_Z \in
 \Omega^2(M)$ is such that $\sigma_Z =\tau^*\bar\sigma_Z$ and $\sigma_Z$ is given in~\eqref{eq:sigmaX}.
\end{lem}
\begin{proof}
See \cite[Prop.\ 2.1]{Asselle:2017}.
\end{proof}

By the previous observation we see that the magnetic flow on $M$ with respect to $(g_M,\bar \sigma_Z)$ lifts to a geodesic flow on $Q$ with respect to the metric $g_Q$ as constructed in Section \ref{sec:mertic}. Let $\bar H:T^*M \to \R$ and $H:T^*Q \to \R$ be the kinetic Hamiltonians with respect to $g_M$ and $g_Q$ respectively. We denote the corresponding Hamiltonian vector fields by $X_{\bar H}$ and $X_H$, where of course $X_{\bar H}$ is defined with respect to the twisted symplectic form $\bar \omega+ \bar \pi^*\bar \sigma_Z$.
\begin{lem}\label{lem:gE}
The geodesic flow on $T^*Q$ with respect to $g_Q$ leaves the subset $A^{-1}(Z)\cap H^{-1}(k)$ invariant and projects to the magnetic flow on $\bar H^{-1}(\bar k) \subset T^*M$ with respect to $(g_M,\bar \sigma_Z)$, where $k=\bar k + \frac 12 |Z|^2$.
\end{lem}
\begin{proof}
See \cite[Lemma 2.2]{Asselle:2017}.
\end{proof}
\noindent In virtue of the previous lemma, any curve $\bar x:\R \to T^*M$ satisfying
\begin{equation}
\dot{\bar x} = X_{\bar H}(\bar x)\,,\quad
x(T) = x(0)\,,\quad
\bar H(x) = \bar k\,,
\label{3.1}
\end{equation}
for some $T>0$, lifts to a curve $x: \R \to T^*Q$ 
\begin{equation}
\dot x = X_{H}(x)\,,\quad
x(T) = g x(0)\,,\quad
H(x) = k\,,\quad
A(x) = Z\,,
\label{3.2}
\end{equation}
for some element $g\in G$. We call such $x$ a \textit{constrained} $G$\textit{-closed geodesic}. Conversely, every constrained $G$-closed geodesic clearly projects to a curve satisfying \eqref{3.1}. 

\begin{rem}
Denote with $\g$ the Lie algebra of $G$, write $\g \to T_qQ, X \mapsto \ul X_q$ and set $\pi_\g:T_qQ \to \g$ to be the orthogonal projection for all $q\in Q$, and finally fix an element $Z \in \g$. Then, we see that the affine subbundle $\D \subset TQ$, given by $\D_q = \{ v \in T_qQ \mid \pi_\g(v)=Z\}$ for all $q\in Q$, is invariant under the geodesic flow of $g_Q$. Furthermore, the projection of a geodesic $\gamma(t)$ satisfying $\dot \gamma(t)\in \D_{\gamma(t)}$, that is a constrained $G$-closed geodesic, is a magnetic geodesic in the quotient $Q/G$ with respect to a two-from $\sigma$ determined by the element $Z$. Conversely, given any magnetic geodesic in $Q/G$ one constructs a lifted curve, which is a geodesic satisfying the constraint. From this it readily follows that the question of the existence of constrained $G$-closed geodesics is strictly related to the field of dynamics of mechanical systems with non-holonomic constraints (cf. \cite{BlochMarsdenMurray:nonholonomic,FedorovJovanovic}
and references therein). Such systems are given by a Lagrangian $L$ defined on $TQ$, the tangent bundle of the configuration space $Q$, and constraints determined by a nonintegrable distribution $\mathcal{D} \subset TQ$. The equations of motion are then derived by the Lagrange d'Alembert principle on curves $\gamma(t)$ that satisfy the constrain $\dot \gamma(t) \in \mathcal{D}_{\gamma(t)}$. A particular case -- known as \emph{(generalized) Chaplygin systems} -- arises when $Q$ is the total space of a principal bundle, $\mathcal{D}$ is the horizontal distribution of a principal connection and the Lagrangian $L$ is invariant under the group action. To the author's best knowledge, even though several authors have been studying the problem of reducing the phase space of such nonholonomic systems, there seems to be no general existence results (such as e.g.\ Theorem \ref{thm:main}) in this setting. 
\end{rem}

Observe that, if  $g=\exp(X)$ for some $X \in \g$, then the rescaled curve 
$$y:\R \to T^*Q, \quad y(t)=\exp(-tX)x(tT)$$ 
satisfies
\begin{equation}
\dot y = - X_{\<A,X\>}(y) + TX_H(y)\,,\quad y(1)=y(0)\,,\quad H(y)=k\,,\quad A(y)=Z\,,
\label{3.3}
\end{equation}
where by $X_{\<A,X\>}$ we denote the Hamiltonian vector field associated with the Hamiltonian function $\<A,X\>:T^*Q \to \R$ and the standard symplectic form. Conversely, 
every loop $y$ satisfying \eqref{3.3} defines a curve $x$ satisfying \eqref{3.2} by simply reversing the scaling. Following \cite[Section 4.2]{Frauenfelder:2008}, we see that such loops satisfying \eqref{3.3} correspond to the critical points of 
$$\A_k:C^\infty(S^1,T^*Q\times \g)\times \R \to \R$$ given by
$$\A_k(y,\phi,T)= \int_0^1 y^*\lambda - \int_0^1 \Big (T(H(y)-k) -\<A(y)-Z,\phi\>\Big ) \, dt,$$
where $\lambda$ denotes the Liouville $1$-form on $T^*Q$. We will not insist further on the functional $\A_k$, since it is not well-suited for finding critical points 
using Morse theory. Nevertheless, we find important to introduce $\A_k$ as motivation on how we will deduce the functional $\Sb_k$ in the next section.

%%%%%%%%%%%%%%%%%%%%%%%%%%%
%%%%%%%%%%%%%%%%%%%%%%%%%%%
%%%%%%%%%%%%%%%%%%%%%%%%%%%

\section{The functional $\Sb_k$}
\label{s:thefunctionalsk}

Let $M$ be a closed orbifold equipped with a Riemannian metric $g_M$ and a closed 2-form $\sigma$. In the previous section we have reformulated the 
closed magnetic geodesics problem in $M$ as the problem of finding $G$-geodesics on a suitable Riemannian locally-free $G$-bundle $(Q,g_Q)$, $G$ compact Lie group, or, equivalently,  critical points of the Rabinowitz-type action functional $\A_k$. Inspired by this, we now define a functional 
over a suitable space of loops in $Q\times \g$, whose critical points correspond precisely to periodic magnetic geodesics in $M$ of fixed energy. 

Thus, fix $k > \frac 12$ and - with the notation introduced in Section~\ref{s:locallyfreeactions} - we define 
\begin{align*}
& \Sb_k: W^{1,2}(S^1,Q)\times L^2(S^1,\g) \times (0,\infty) \to \R, \\
& \Sb_k(\gamma,\phi,T) = \frac 1 {2T} \int_0^1 |\dot \gamma(t) +  \ul{\phi(t)}(\gamma(t))|^2 \, dt + \int_0^1 \<\phi(t),Z\>\, dt + kT\,.
\end{align*}
where $\<\cdot,\cdot \>$ is any $\Ad$-invariant metric on $\g$ such that $\<Z,Z\>=1$ and $\ul{\phi(t)}$ denotes the fundamental vector field associated with the Lie algebra element $\phi(t)$. For notational convenience 
we will hereafter omit the $t$-dependence everywhere. 
We set 
\[\mathcal M:= W^{1,2}(S^1,Q)\times L^2(S^1,\g)\times(0,+\infty)\,,\] and observe that $\mathcal M$ has a natural structure of (non-complete) product Hilbert manifold; 
we denote with $g_{\mathcal M}$ the product metric.
We also notice that the functional $\Sb_k$ is smooth on $\M$ and that the connected components of $\mathcal M$ are in one-to-one correspondence with conjugacy classes in $\pi_1(Q)$. 

\begin{rmk}
The functional $\Sb_k$ can be thought of as the Legendre dual of the functional $\mathbb A_k$ introduced in Section 4. Indeed, computing for every fixed $(\phi,T)$ the Lagrangian $L_{\phi,T}$ which is the Fenchel dual of the Hamiltonian $H_{\phi,T}:= T\cdot H - \<A,\phi\>$ and then letting $(\phi,T)$ free yields precisely the functional $\Sb_k$ above. 
\end{rmk}

\begin{lem}
If $(\gamma,\phi,T)$ is a critical point of $\Sb_k$, then the curve 
\begin{equation}
\label{eq:correspondence}
t \mapsto \tau(\gamma(t/T))
\end{equation}
is a closed magnetic geodesic with energy $k-\frac 12$ and period $T$ for the magnetic flow on $T^*M$ defined by the Riemannian metric $g_M$ and the 2-form $\sigma$.
\label{lem:critpoints}
\end{lem}
\begin{proof}
 By rescaling it is enough to show that the loop $\bar \gamma= \tau \circ \gamma$ is a magnetic geodesic for $T \sigma$ with energy $T^2(k-1/2)$. We have for all $
\xi \in W^{1,2}(\gamma^*TQ)$ and $\eta \in L^2(S^1,\g)$
\begin{equation}\label{eq:system}
\begin{aligned}
0 &= d\Sb_k(\gamma,\phi,T)[0,0,1]=  -\frac 1 {2T^2} \int_0^1 |\dot \gamma + \ul \phi(\gamma)|^2\, dt + k\\
0 &= d\Sb_k (\gamma,\phi,T)[0,\eta,0] = \frac 1 T \int_0^1 \<\dot \gamma +\ul \phi(\gamma), \ul\eta\>\, dt + \int_0^1 \<\eta,Z\>\, dt\\
0&= d\Sb_k(\gamma,\phi,T)[\xi,0,0] = \frac 1T \int_0^1 \<\dot \gamma+ \ul \phi(\gamma), \nabla_t \xi + \nabla_\xi \ul\phi(\gamma)\> \, dt .
\end{aligned}
\end{equation}

By assumption $\dot \gamma \in L^2(\gamma^*TQ)$ and hence the function $\psi:t \mapsto \theta_{\gamma(t)}(\dot \gamma(t)) \in \g$ is defined almost everywhere. This shows that for almost all $t \in S^1$ the vertical part of $\dot \gamma(t)$ is the fundamental vector field of $\psi(t)$ at $\gamma(t)$. We reformulate the second equation of~\eqref{eq:system} as 
\[
0 = \int_0^1 \<\dot \gamma + \ul \phi(\gamma) + T\ul Z(\gamma),\ul \eta(\gamma)\>\, dt= \int_0^1 \<\psi + \phi + TZ,\eta\>\, dt\,.
\]
Therefore,
\[
\psi(t)+\phi(t)=-TZ
\]
for almost all $t \in S^1$.
We now show that $\bar \gamma=\tau \circ \gamma$ is a magnetic geodesic for $T\sigma$. Fix any $t_0\in S^1$, pick charts $(U,0,\vp)$ and $(U',\Gamma',\vp')$ about $\gamma(t_0)$ and $\bar \gamma(t_0)$ respectively and let $\tilde \tau:U \to U'$ be the local representative of $\tau$ in the sense of Definition~\ref{dfn:orbismooth}. There exists $\e>0$ such that $\gamma(t)\in U$ for all $t\in I:=(t_0-\e,t_0+\e)$. We need to check that the curve $\tilde \gamma=\tilde \tau \circ \gamma|_I:I \to U'$ is smooth and satisfies  
\begin{equation}\label{eq:Tmagnetic}
\na_t \tilde \gamma = TY_{\tilde \gamma}(\dot {\tilde \gamma})\,,
\end{equation}
where $Y$ is definded via $\tilde \sigma_p(u,v) = \< Y_p(u),v\>$ for all $p \in U'$ and $u,v \in T_p U'$. Given any $\tilde \xi \in W^{1,2}(\tilde \gamma^*TU')$ let $\xi = \tilde \xi^\hor \in W^{1,2}((\gamma|_I)^*TQ)$ be the horizontal lift. We claim that it suffices to show that for almost all $t \in I$ we have 
\begin{equation}\label{eq:main}
\<\dot \gamma + \ul \phi(\gamma),\na_t \xi + \na_{\xi} \ul \phi(\gamma)\> = \<\dot{\tilde \gamma}, \na_t \tilde \xi \> + T \tilde \sigma( \dot{\tilde \gamma}, \tilde \xi)\,,
\end{equation}
where $\tilde \sigma$ is the local representative of $\sigma$. Indeed, integrating \eqref{eq:main} over $I$ and using the last equation in \eqref{eq:system}, we get for all $\tilde \xi$ with compact support
\begin{align*}
0 &= \int_0^1\<\dot \gamma + \ul \phi(\gamma),\na_t \xi + \na_{\xi} \ul \phi(\gamma)\>\, dt \\ &= \int_I \ \<\dot \gamma + \ul \phi(\gamma),\na_t \xi + \na_{\xi} \ul \phi(\gamma)\>\, dt\\
&= \int_I \big (\<\dot{\tilde \gamma}, \na_t \tilde \xi \> + T \tilde \sigma( \dot{\tilde \gamma}, \tilde \xi)\big )\, dt\\
&=\int_I \big(\<\dot{\tilde \gamma},\na_t \tilde \xi\> + T\<Y_\gamma(\dot{\tilde \gamma}),\tilde \xi\>\big)\, dt\,,
\end{align*}
where we have extended $\xi$ by zero outside $I$. By elliptic bootstrapping we conclude that $\tilde \gamma$ is smooth and  solves \eqref{eq:Tmagnetic}. Hence we are left to show~\eqref{eq:main}. First, one checks directly that the function mapping $
\tilde \xi$ to ``right-hand sinde of \eqref{eq:main} minus left-hand side of \eqref{eq:main}'' is linear over the ring $W^{1,2}(S^1,\R)$.  Hence, we can assume without loss of generality that $\tilde \xi$ is a pull-back of a smooth vector field on $U'$ and $\xi$ is the pull-back of the horizontal lift of that vector field. By abuse of notation we denote these vector fields by the same symbol.
For sake of simplicity, the following computation will be made pointwise, even though everything makes sense only almost everywhere. Also, we omit the reference to $t$ in the notation. We compute
\[\<\dot \gamma+\ul \phi,\na_t \xi + \na_\xi \ul \phi\> =\<\dot \gamma,\na_t \xi\> + \<\dot \gamma,\na_\xi \ul \phi \> + \<\ul \phi, \na_t \xi\>+\<\ul \phi,\na_\xi\ul\phi\>\,.\]
The function $|\ul \phi|^2=|\phi|^2$ is (for fixed $t\in I$) constant on $Q$ and hence $0=\xi|\ul \phi|^2 = 2\<\ul \phi,\na_\xi \ul \phi\>$. We split $\dot \gamma(t)=u(t)+v(t)$ into horizontal and vertical vector. Recall that $v(t)$ is the fundamental vector field associated to $\psi(t)$ at $\gamma(t)$; therefore
\[
\<\dot \gamma+\ul \phi,\na_t \xi + \na_\xi \ul \phi\> = \<u,\na_t \xi\> + \<\dot \gamma,\na_\xi (\ul \phi+\ul \psi)\> + \<\ul \psi+\ul\phi,\na_t \xi\>-\<\dot \gamma,\na_\xi \ul \psi\>\,.
\]
Using the fact that $\psi+\phi=-TZ$ we obtain
\[\<\dot \gamma+\ul \phi,\na_t \xi + \na_\xi \ul \phi\> = \<u,\na_t \xi\> - T\<\dot \gamma,\na_{\xi} \ul Z\> - T\<\ul Z,\na_t \xi\> - \<\dot \gamma,\na_\xi \ul\psi\>\,.
\]
Now use the fact that $\xi$ is horizontal by construction and Lemma~\ref{lem:antisymmetric} to infer 
\begin{align*}
\<\dot \gamma+\ul \phi,\na_t \xi + \na_\xi \ul \phi\> &=\<u,\na_t \xi\>- T\<\dot \gamma,\na_\xi \ul Z\>+T\<\na_{\dot \gamma} \ul Z,\xi\>-T\pt\<\ul Z,\xi\>-\<\dot \gamma,\na_\xi \ul \psi\>\\
&=\<u,\na_t \xi\>- 2T\<\dot \gamma,\na_\xi \ul Z\>-\<\dot \gamma,\na_\xi \ul \psi\>\,.
\end{align*}
and Lemma~\ref{lem:sigma} together with the fact that $\na_t \xi =\na_{\dot \gamma} \xi$ to obtain
\begin{align*}
\<\dot \gamma+\ul \phi,\na_t \xi + \na_\xi \ul \phi\> &= \<u,\na_u \xi\>+\<u,\na_{\ul \psi} \xi\>+T\sigma_Z(\dot \gamma,\xi)-\<u,\na_\xi \ul \psi\>\\
	&=\<u,\na_u \xi\> +T\sigma_Z(\dot \gamma, \xi\> - \<u,[\psi,\xi]\>.
\end{align*}
Since the horizontal lift is invariant under the flow of $\ul \psi$, we conclude that $[\ul \psi,\xi]$ and hence the last term vanishes. Also, since by construction 
$$\sigma_Z(\dot \gamma,\xi) = \tilde \sigma(\dot{\tilde \gamma}, \tilde \xi)$$ 
(cf.\ Proposition~\ref{prop:orb}), it remains only to show that
\begin{equation}\label{eq:claim}
\<u,\na_u \xi\>= \<\dot{\tilde \gamma},\na_t \tilde \xi\>\,.
\end{equation}
For every $t \in I$ extend $\dot{\tilde \gamma}(t)$ to a vector field, denoted $\tilde u$, and let $u$ be its horizontal lift.
We compute 
\[
\<u,\na_u \xi\>=\<u,\na_\xi u \> + \<u,[u,\xi]\> = \frac 12 \xi (|u|^2)+ \<u,[u,\xi]\>\,,
\]
and similarly 
\[
\<\dot{\tilde \gamma},\na_t \tilde \xi\>=
\<\tilde u ,\na_{\tilde u} \tilde \xi\> = \frac 12  \tilde \xi(|\tilde u|^2) +\<\tilde u,[\tilde u,\tilde \xi]\>\,.
\]
Using the chain-rule, the functoriality properties of the Lie bracket and the specific choice of the metric, it is easy to check that the right-hand sides of the last two equations are equal. We conclude~\eqref{eq:claim} and thus~\eqref{eq:main}.

We come to the last statement. Since horizontal and vertical components are orthogonal by construction we have
\[
|\dot \gamma+ \ul \phi(\gamma)|^2 = |\dot {\tilde \gamma}|^2 +|\psi + \phi|^2 = |\dot{\tilde \gamma}|^2 +T^2\,.
\]
Now since $\tilde \gamma$ is a magnetic geodesic the function $t \mapsto |\dot{\tilde\gamma}(t)|$ is constant and so is $t \mapsto |\dot \gamma(t)+
\ul{\phi(t)}(\gamma(t))|$. Together with the first equation of~\eqref{eq:system} we get
\[
|\dot {\tilde \gamma}|^2 + T^2 = |\dot \gamma+\ul\phi(\gamma)|^2 = 2T^2k\,.
\]
This shows the energy claim. 
\end{proof}

\begin{rmk}
If $G$ is abelian then the fundamental vector fields associated to (constant) Lie algebra elements give already enough symmetry to infer that critical points of $\Sb_k$ project to closed magnetic geodesics in $M$. In fact, in this setting closed magnetic geodesics turn out to correspond to critical points of the functional $\Sb_k:W^{1,2}(S^1,Q)\times \g \times (0,+\infty)\to \R$ given by
$$\Sb_k(\gamma,X,T) = \frac1{2T}\int_0^1 |\dot \gamma + \ul X(\gamma)|^2\, dt + \<X,Z\> +kT.$$
In the special case $G=S^1$ we retrieve precisely the functional considered in \cite{Asselle:2017}.
\end{rmk}

%%%%%%%%%%%%%%%%%%%%%%%%%%%%%%%%%%%%%

\subsection{The gauge group} The domain of $\Sb_k$ has many degrees of freedom which do not play any role for the magnetic geodesic obtained in the quotient. By the same token, different critical points of $\Sb_k$ might correspond to the same periodic magnetic geodesic in $M$ via~\eqref{eq:correspondence}. To remedy this fact we introduce on $\M$ the action of a group $\G$  which leaves the critical points set invariant. The group is the \textit{loop group} 
\[\G=W^{1,2}(S^1,G)\]
with group law given by pointwise multiplication and the action of $\rho \in \G$ on $\M$ is given by 
\[
\rho\cdot (\gamma,\phi,T) = (\rho \gamma,\Ad_\rho \phi - \pt \rho \rho^{-1},T),\quad \forall\ (\gamma,\phi,T) \in \M\,.
\]
\begin{lem}\label{lem:gauge}
If $(\gamma,\phi,T)$ is a critical point of $\Sb_k$, then $\rho\cdot(\gamma,\phi,T)$ is also a critical point of $\Sb_k$ for any $\rho \in \G$. 
\end{lem}
\begin{proof}
Set $\gamma^\rho:=\rho \gamma$ and $\phi^\rho:=\Ad_\rho \phi -\pt \rho \rho^{-1}$. We compute directly $\dot \gamma^\rho = d\rho \dot \gamma +\ul{\pt \rho \rho^{-1}}$, where $d\rho$ denotes the differential of the $\rho$-action. Using the identity 
\begin{equation}\label{eq:Adidentity}
\ul{\Ad_\rho \phi} = d\rho \ul \phi\,,
\end{equation}
we conclude that $\dot \gamma^\rho + \ul \phi^\rho = d\rho(\dot \gamma + \ul \phi)$ almost everywhere. In particular 
\begin{equation}
\label{eq:gammarho}
\int_0^1 |\dot \gamma+\ul \phi(\gamma)|^2\, dt = \int_0^1|\dot \gamma^\rho + \ul \phi^\rho(\gamma^\rho)|^2\, dt\,.
\end{equation}
Moreover 
$$\int_0^1 \<\dot \gamma^\rho +\ul \phi^\rho(\gamma^\rho),\ul \eta(\gamma^\rho)\>\, dt   = \int_0^1 \<\dot \gamma+\ul \phi(\gamma),d\rho^{-1} \ul \eta(\gamma)\>\, dt, \quad \forall \eta\in L^2 (S^1,\g)$$
and, since $Z$ is central,
$$\int_0^1 \<\Ad_\rho^{-1} \eta,Z\>\, dt =\int_0^1 \<\eta,Z\>\, dt.$$ 
This shows that the first two equations of~\eqref{eq:system} are satisfied if $\gamma$ and $\phi$ are replaced by $\gamma^\rho$ and $\phi^\rho$ respectively. We now consider the last equation in \eqref{eq:system}. Given a vector field $\xi$ in $\gamma^*TQ$, we consider a map $u:(-\e,\e)\times S^1 \to Q$ such that $u(0,\cdot)=\gamma$ and $\ps u(s,\cdot)|_{s=0}=\xi$. We set $\gamma_s:=u(s,\cdot)$ and $\gamma_s^\rho = \rho \gamma^s$ for all $s\in (-\e,\e)$ and observe that, by the same computation as above,
$$|\dot \gamma^\rho_s +\ul \phi^\rho|=|\dot \gamma_s+\ul \phi|,\quad \forall s\in (-\e,\e).$$ 
Thus
\[
\Sb_k(\gamma_s,\phi,T) = \Sb_k(\gamma_s^\rho,\phi^\rho,T)- \Delta_\rho,\qquad \forall\ s \in (-\e,\e)\,,
\]
where 
\begin{equation}
\label{eq:Delta}
\Delta_\rho:=\int_0^1 \<\pt \rho\rho^{-1},Z\>\, dt
\end{equation}
is independent of $s$. Differentiating in $s$ and evaluating at $s=0$ yields
\[
0=d_{\gamma}\Sb_k(\gamma,\phi,T)[\xi] = d_{\gamma^\rho} \Sb_k(\gamma^\rho,\phi^\rho,T)[d\rho \xi]\,.\qedhere
\]
\end{proof}

\vspace{2mm}

The computations above show that the functional $\Sb_k$ is not invariant under the action of $\G$, although the set of critical points does. In fact, we have
\begin{equation}
\label{eq:levelaction}
\Sb_k(\rho\cdot (\gamma,\phi,T))= \Sb_k(\gamma,\phi,T)+\Delta_\rho,\quad \forall \rho\in \G,\ \forall (\gamma,\phi,T) \in \M\,,
\end{equation}
where $\Delta_\rho$ is given by \eqref{eq:Delta}. In case $\rho =\exp(\eta)$ for some path $\eta:[0,1] \to \g$, we have
\begin{equation}
\Delta_\rho=\<\eta(1)-\eta(0),Z\>.
\label{eq:easierdelta}
\end{equation}
Indeed, $\pt \rho \rho^{-1} = \Ad_\rho \pt \eta$ almost everywhere and hence 
$$\int_0^1 \<\pt \rho \rho^{-1},Z\>\, dt =\int_0^1 \<\pt \eta,Z\>\, dt$$ 
since $Z$ is central. The claim follows then 
by the fundamental theorem of calculus. It is worth to notice that, if $G$ is not connected, then not every $\rho\in \G$ can be written as $\exp(\eta)$ 
for some $\eta:[0,1]\rightarrow \g$. However, for our purposes, it will always be sufficient to consider elements $\rho$ which admit such a representation.

From \eqref{eq:levelaction} we can also deduce that $\Sb_k$ is not bounded from above nor from below. 
Indeed, consider $X \in \g$ such that $\<X,Z\>\neq 0$ and $\exp(X)=e$, where $e$ is the neutral element of $G$. For every $m\in \Z$ define $\rho_m \in \G$ via $\rho_m(t) = \exp(mX t)$. By~\eqref{eq:levelaction} we then get
$$\Sb_k(\rho_m\cdot (\gamma,\phi,T)) = \Sb_k(\gamma,\phi,T) + m \<X,Z\>.$$ 

%%%%%%%%%%%%%%%%%%%%%%%%%%%%%%%%%%%%%%%%%

\subsection{The Palais-Smale condition up to gauge transformations.} In infinite dimensional Morse theory, the Palais-Smale condition plays  the role of compactness, since it roughly speaking allows to find critical points of a functional on a Hilbert manifold from a sequence of approximately critical points. The lack of such a compactness property poses therefore major difficulties and one is forced to look for additional informations in order to prove existence of critical points. An evidence of this is represented precisely by the functional $\Sb_k:\mathcal M\rightarrow \R$. Indeed, in the case of $S^1$-actions treated in \cite{Asselle:2017}, the Palais-Smale condition for $\Sb_k$ does not hold on $\mathcal M$, but rather on subsets $\mathcal M_{[T_*,T^*]}\subset \M$ of triples $(\gamma,X,T)$ with $0<T_*\leq T\leq T^*$. As it turns out, this is enough to show existence of critical points of $\Sb_k$ - for almost every $k$ - by means of a clever monotonicity argument, better known as the \textit{Struwe monotonicity argument} \cite{Struwe:1990sd} (for other applications we refer e.g.\ to \cite{Abbondandolo:2013is,Abbondandolo:2017,Abbondandolo:2014rb,Abbondandolo:2015lt,Asselle:2016jfa,Asselle:2015ij,Asselle:2014hc,Asselle:2015hc,Contreras:2006yo}).

In the setting considered in this paper the situation becomes even more delicate, since if $G \not = S^1$ the functional $\Sb_k$ fails to satisfy the Palais-Smale condition even on the subsets $\M_{[T_*,T^*]}$. 	Nevertheless, a suitable generalization of the Palais-Smale condition (namely, the Palais-Smale condition ``up to gauge transformations'') for $\Sb_k$ on $\M_{[T_*,T^*]}$ turns out to hold true. Recall that a sequence $(\gamma_h,\phi_h,T_h)_{h \in \N}\subseteq \M$ is called a \emph{Palais-Smale sequence for $\Sb_k$} if there exists $c \in \R$ such that 
\[\lim_{h\rightarrow +\infty}\ \Sb_k(\gamma_h,\phi_h,T_h)=c\, , \qquad \lim_{h\rightarrow +\infty}\ |d\Sb_k(\gamma_h,\phi_h,T_h)| =0\,.\]
More precisely, we say that $(\gamma_h,\phi_h,T_h)_h$ is a \emph{Palais-Smale sequence for $\Sb_k$ at $c$}. 
In the definition above $|\cdot |$ denotes, with slight abuse of notation, the (dual) norm on $T^*\mathcal M$ induced by the metric $g_{\mathcal M}$. The functional $\Sb_k$ is said 
to \textit{satisfy the Palais-Smale condition} if every Palais-Smale sequence admits a converging subsequence.

\begin{dfn}
A sequence $(\rho_h)\subset \G$ is called \emph{admissible}, if the sequence $(\Delta_{\rho_h})_h \subset \R $, $\Delta_{\rho_h}$ defined as in \eqref{eq:Delta} for every $h\in \N$, is bounded from above and below.
\end{dfn}

Let $(\gamma_h,\phi_h,T_h)_{h} \subset \M$ be a Palais-Smale sequence for $\Sb_k$ and $(\rho_h)_{h \in \N} \subset \G$ be an admissible sequence. Then, up to passing to a subsequence if necessary, $(\rho_h\cdot (\gamma_h,\phi_h,T_h))_{h}$ is again a Palais-Smale sequence for $\Sb_k$. Indeed, 
after passing to a subsequence we can assume that $\Delta_{\rho_h}\to \Delta$ converges. That $(\rho_h\cdot (\gamma_h,\phi_h,T_h))_h$ is a Palais-Smale sequence at level $c+\Delta$ follows now from the same computations as in Lemma~\ref{lem:gauge} and Equation~\eqref{eq:levelaction}. If $G \neq S^1$, then it is easy to construct an admissible sequence $(\rho_h) \subset \G$ such that $(\rho_h\cdot (\gamma_h,\phi_h,T_h))$ does not contain a converging subsequence, even though $(\gamma_h,\phi_h,T_h)\subseteq \M_{[T_*,T^*]}$. Observe that the gauged sequence still belongs to $\M_{[T_*,T^*]}$, since $\M_{[T_*,T^*]}$ is obviously $\G$-invariant. Therefore, $\Sb_k$ does not satisfy the Palais-Smale condition on $\M_{[T_*,T^*]}$.

On the other hand, given a Palais-Smale sequence  $(\gamma_h,\phi_h,T_h)\subseteq \M_{[T_*,T^*]}$, one might hope to find an admissible sequence $(\rho_h)\subseteq \G$
such that $(\rho_h\cdot (\gamma_h,\phi_h,T_h))$ has a converging subsequence. By this observation we are naturally led to the following

\begin{dfn}
A $\G$-invariant subset $\K \subset \M$ is said to \emph{satisfy the Palais-Smale condition up to gauge transformations} if for any Palais-Smale sequence $(\gamma_h,\phi_h,T_h)_h \subset \K$ there exists an admissible sequence $(\rho_h)\subset \G$ such that $(\rho_h\cdot (\gamma_h,\phi_h,T_h))_h$ contains a converging subsequence. 
\end{dfn}

Before we come to the main result of this subsection, namely that $\M_{[T_*,T^*]}$ does satisfy the Palais-Smale condition up to gauge transformations, we prove some elementary estimates for Palais-Smale sequences which will be useful later on. 
\begin{lem}\label{lem:Tvp}
Suppose $(\gamma_h,\phi_h,T_h)$ is a Palais-Smale sequence for $\Sb_k$ at level $c$, then
\begin{equation}\label{eq:Tvp}
\int_0^1 |\dot \gamma_h +\ul{\phi_h}(\gamma_h)|^2\, dt = O(T^2_h)\,,\quad
\left|\int_0^1 \<\phi_h,Z\>\, dt\right|=O(T_h).
\end{equation}
Moreover, $T_h\rightarrow 0$ if and only if $\displaystyle \int_0^1 \<\phi_h,Z\>\, dt\rightarrow c$. In this case we have 
$$\int_0^1 |\dot \gamma_h +\ul{\phi_h}(\gamma_h)|^2\, dt \to 0.$$
\end{lem}
\begin{proof}
  We have 
\begin{align}
c + o(1)&  =\  \frac{1}{2T_h}\int_0^1 |\dot \gamma_h+\ul{\phi_h}(\gamma_h)|^2\, dt + \int_0^1 \<\phi_h,Z\>\, dt + kT_h,\label{1}\\ 
o(1) & = \ \frac{\partial \Sb_k}{\partial T}(\gamma_h,\phi_h,T_h) \ =\  k -\frac{1}{2T_h^2}\int_0^1 |\dot \gamma_h+\ul{\phi_h}(\gamma_h)|^2\, dt.\label{2}
\end{align}
From \eqref{2} it follows that 
$$\frac{1}{2T_h} \int_0^1 |\dot \gamma_h+\ul{\phi_h}(\gamma_h)|^2 \, dt \ =\ kT_h + T_ho(1)\,.$$
The first equation in \eqref{eq:Tvp} follows. Replacing the last expression in \eqref{1} we get
\[ \int_0^1 \<\phi_h,Z\>\, dt = c-2k T_h  -T_h o(1) + o(1)\,.\]
This shows the rest of the claim.
\end{proof}

%\begin{lem}
%Suppose $(\gamma_h,\phi_h,T_h)$ is a Palais-Smale sequence for $\Sb_k$ at level $c$ such that $T_h \to 0$ and set $\bar \gamma_h := \tau \circ \gamma_h$. Then we have
%$$\int_0^1 |\dot \gamma_h+\ul{\phi_h}(\gamma_h)|^2\, dt \to 0, \quad \int_0^1 |\dot{\bar \gamma}_h|^2\, dt \to 0.$$
%\end{lem}
%\begin{proof}
%The first assertion follows trivially from
 % \eqref{2}. Consider now for every $h\in \N$ the splitting
 % $$\dot \gamma_h= u_h+v_h$$ 
%into horizontal and vertical directions. Using again \eqref{2} we get
 % \begin{align*}
  %  2kT_h^2 + o(T_h^2) & = \int_0^1 \left| u_h +v_h+ \ul{\phi_h}(\gamma_h)\right|^2 \, dt = \int_0^1 |u_h|^2\, dt + \int_0^1 \left| v_h+ \ul{\phi_h}(\gamma_h)\right|^2 \, dt.
 % \end{align*}
 % In particular
%$$\int_0^1 |\dot{\bar \gamma}_h|^2\, dt= \int_0^1 |u_h|^2\, dt = o(1),$$
%as by construction $d\tau$ is an isometry on $\ker \theta$.
%\end{proof}

\begin{lem}\label{lem:phibded}
Suppose $(\gamma_h,\phi_h,T_h)_h\subseteq \M$ is a Palais-Smale sequence for $\Sb_k$ such that $T_h =O(1)$. Then there exists an admissible sequence $(\rho_h) \subset \G$ such that the gauged sequence $(\rho_h\cdot \phi_h =\Ad_{\rho_h}\phi_h-\pt \rho_h\rho_h^{-1} )_h$ admits a (strongly in $L^2$) converging subsequence.
\end{lem}
\begin{proof}
We divide the proof in two steps:

\textbf{Step 1:} Fix a maximal torus $T \subset G$ (that is, a connected, closed and abelian subgroup of maximal dimension) with Lie algebra $\t$. For every $h \in \N$ we define 
\[\ol \phi_h :=\int_0^1 \phi_h(t)\, dt \in \g\,.\]
 By \cite[Thm.\ 6.4]{Bredon:Trafo} there exists $g_h\in G$ such that $\Ad_{g_h}\ol\phi_h \in \t$. Set $\Lambda :=\{ X \in \t \mid \exp(X)=e\}$, where as usual $e$ denotes the neutral element in $G$. The subgroup $\Lambda \subset \t$ is a lattice and so we find $X'_h \in \Lambda$ such that $|\Ad_{g_h} \ol\phi_h-X'_h|$ is uniformly bounded. With $X_h:=\Ad_{g_h}^{-1}X'_h$ we now have
\begin{equation}\label{eq:XYbded}
\big |\ol\phi_h-X_h\big |=O(1)\,.
\end{equation}
Now for all $h\in \N$ define $\eta_h:[0,1]\to \g$ and $\rho_h:[0,1] \to G$ by
\[
 \eta_h(t):=\int_0^t \phi_h(s)\, ds - t(\ol \phi_h-X_h),\quad \rho_h(t):=\exp(\eta_h(t)).
\]
Notice that  $\eta_h(0)=0$ and $\eta_h(1)=X_h$. Therefore, $\rho_h\in \G$ for all $h\in\N$ as by construction $\rho_h$ is of class $W^{1,2}$ and $\exp(X_h)=e$. 
Gauging $\phi_h$ by $\rho_h$ yields
\begin{equation}\label{eq::rhoh}
\rho_h \cdot \phi_h = \Ad_{\rho_h} \phi_h - \partial_t \rho_h \rho_h^{-1} = \Ad_{\rho_h}\phi_h - \Ad_{\rho_h} \partial_t \eta_h = \Ad_{\rho_h}(\ol \phi_h-X_h);
\end{equation}
in particular, the gauged sequence $(\rho_h\cdot \phi_h)$ is contained in $W^{1,2}(S^1,\g)$, as $\Ad$ is smooth and $\ol \phi_h - X_h$ is constant. Moreover, we have the pointwise
estimate
\[
\big |\rho_h\cdot \phi_h\big|=\big |\Ad_{\rho_h}\phi_h-\pt \rho_h\rho_h^{-1}\big |=\big |\phi_h-\pt \eta_h\big |=\big |\ol \phi_h-X_h\big  |= O(1)\,.
\]
This shows that the gauged sequence is uniformly bounded in $L^\infty$ and hence also in $L^2$. Thus, it only remains to to check that the sequence $(\rho_h)_h$ is admissible. By Lemma~\ref{lem:Tvp} we have
$$\<\ol\phi_h,Z\> = \int_0^1 \<\phi_h(t),Z\>\, dt = O(T_h) = O(1).$$
and thus by \eqref{eq:easierdelta}  and \eqref{eq:XYbded} 
\[
\big |\Delta_{\rho_h}\big | = \big |\<X_h,Z\>\big |\leq\big |X_h-\ol\phi_h\big | +\big |\<\ol\phi_h,Z\>\big | = O(1)\,.
\]

\vspace{2mm}

\textbf{Step 2:} By the first step, after gauging and passing to a subsequence if necessary, we can assume that $(\gamma_h,\phi_h,T_h)$ is a Palais-Smale sequence for $\Sb_k$ with $(\phi_h)\subseteq W^{1,2}(S^1,\g)$ uniformly bounded in $L^\infty$. We now repeat the gauge procedure as in Step 1 taking $X_h=0$ and observe that the gauged sequence 
$(\rho_h\cdot \phi_h)$ is contained in $W^{2,2}(S^1,\g)$. Therefore, we can compute the derivative with respect to $t$ of $\rho_h\cdot \phi_h$ and, using~\eqref{eq::rhoh}, we obtain
%form again $\phi_h^\rho=\Ad_{\rho_h}\phi_h - \pt \rho_h \rho^{-1}_h=\Ad_{\rho_h}(\phi_h-\pt \eta)=\Ad_{\rho_h}(\ol{\phi}_h)$. Deriving by $\pt$ we get
\[
\pt (\rho_h \cdot \phi_h) = \Ad_{\rho_h}\big (\ad_{\pt \eta_h}\big (\ol\phi_h\big )\big ).
\]
In particular we obtain the pointwise estimate
\[
\big |\pt (\rho_h \cdot \phi_h)\big | = \big |\ad_{\pt \eta_h}\ol\phi_h\big |\leq \big |\pt \eta_h\big | \big |\ol \phi_h\big | \leq \big (|\phi_h| + |\ol\phi_h|\big )|\ol\phi_h| = O(1)
\]
which shows that $(\rho_h\cdot \phi_h)_h$ is uniformly bounded in $W^{1,\infty}$, hence in $W^{1,2}$. The claim follows using the compactness of the embedding $W^{1,2} \hookrightarrow L^2$. 
\end{proof}

\begin{rem}
In the proof of the lemma above we actually showed that we can choose the admissible sequence $(\rho_h)\subseteq \mathcal G$ in such a way that the gauged sequence $(\rho_h\cdot \phi_h)$ 
is uniformly bounded in $W^{1,2}(S^1,\g)$ and hence, in particular admits, a weakly (in $W^{1,2}$) converging subsequence. 
\end{rem}

\begin{lem}\label{lem:palaissmale}
Suppose $(\gamma_h,\phi_h,T_h)$ is a Palais-Smale sequence for $\Sb_k$ such that $0<T_*\leq T_h\leq T^*$ for all $h\in \N$, then there exists an admissible sequence $(\rho_h)\subset \G$ such that $(\rho_h\cdot (\gamma_h,\phi_h,T_h))$ contains a strongly converging subsequence. 
\end{lem}
\begin{proof} 
By Lemma~\ref{lem:phibded}, up to taking a subsequence and applying a gauge transformation we can assume that $\phi_h$ converges weakly in $W^{1,2}$ (and strongly in $L^2$) to some $\phi\in W^{1,2}(S^1,\g)$, and that $T_h \to T \in[T_*,T^*]$ as $h\to \infty$. Using the inequality $(a+b)^2\leq 2a^2+2b^2$ and Lemma~\ref{lem:Tvp} we obtain 
\begin{equation}\label{eq:dotgammasmall}
\int_0^1 |\dot \gamma_h|^2\, dt \leq 2 \int_0^1 |\dot \gamma_h + \ul{\phi_h}(\gamma_h)|^2 \, dt + 2\int_0^1 |\phi_h|^2\, dt = O(1).
\end{equation}
This shows that $\|\dot \gamma_h\|_{2}$ is uniformly bounded and, hence, that the sequence $(\gamma_h)_h$ is $\frac12$-H\"older equicontinuous. Up to taking a subsequence, the theorem of Arzela-Ascoli yields the existence of $\gamma \in C^0(S^1,E)$ such that  $\gamma_h \to \gamma$ uniformly as $h\to \infty$. The fact that the convergence of $\gamma_h$ to $\gamma$ is actually strong in $W^{1,2}$ follows now by standard arguments, as we show  in Lemma \ref{lem:palaissmale2} (see also \cite[Lemma 5.3]{Abbondandolo:2013is}).
\end{proof}

%%%%%%%%%%%%%%%%%%%%%%%%%%%%%%%%

\subsection{A complete gradient vector field for $\Sb_k$.}\label{sec:completeflow} Consider the bounded vector field 
\begin{equation}\label{boundedvf}
\mathcal{X}_k:= \frac{-\grad \Sb_k}{\sqrt{1+|\grad \Sb_k|^2}}
\end{equation}
conformally equivalent to $-\grad \Sb_k$, where the gradient of $\Sb_k$ is defined with respect to the Riemannian metric $g_{\mathcal M}$. Since $\Sb_k$ is smooth, the vector field $\mathcal{X}_k$ is locally Lipschitz continuous and hence its flow $\Phi_k$ is well-defined. However, $\Phi_k$ is not complete, since there are flow-lines on which the variable $T$ approaches zero in finite time. On the other hand, the only source of non completeness for $\Phi_k$ is represented by such flow-lines; hence, ``stopping them'' in a suitable fashion yields a complete flow. However, while doing this we should be careful not to lose any geometric property of the functional $\Sb_k$. To this purpose, we need to know more about the behavior of the functional $\Sb_k$ in a neighborhood of elements that are approached by finite maximal flow-lines on which $T\to 0$. 

We call an element $(\gamma,\phi,T) 
\in \M$  \emph{vertical}, if $\dot \gamma+\ul \phi(\gamma)=0$ almost everywhere. We see that if $(\gamma,\phi,T)$ is vertical then necessarily the vector $\dot \gamma(t)$ is vertical for almost every $t \in S^1$ and $\gamma$ projects to a constant loop in the quotient $M$. Moreover by~\eqref{eq:levelaction} we see  immediately that 
\begin{equation}
\label{eq:levelatvertical}
\Sb_k(\gamma,\phi,T)= \int_0^1\<\phi,Z\>\, dt + kT\,.
\end{equation}
We now examine neighborhoods of vertical elements in $\M$. For $\delta>0$ define
\[
\V_\delta := \left\{ (\gamma,\phi,T) \in \M \ \Big |\  \int_0^1 \left|\dot \gamma + \ul \phi(\gamma)|^2\, dt <\delta\right.\right\}\,.
\]
Our first goal is to show that, for $\delta >0$ sufficiently small, the space $\V_\delta$ is a disjoint union of neighborhoods of ``local minima'' of $\Sb_k$.
To this purpose we need some notation: we set $Z(G) := \{g \in G \mid gh=hg,\, \forall\, h \in G\}$ to be the center of $G$, $\z$ to be its Lie algebra and $p_\z:\g\to \z$ to be the orthogonal projection. Note that $\z$ is non-trivial because $Z\in \z$ is non-trivial by Proposition~\ref{prop:orb}.  
Further we denote with $\Lambda_\z := \{ X \in \z \mid \exp(X)=e\}$ the unit lattice in $\z$, where $e \in G$ is the neutral element of $G$. Finally, for every $\phi\in L^2(S^1,\g)$ we define
 \begin{equation}
 \ol \phi:=\int_0^1 \phi(t)\, dt\in \g.
 \label{olphi}
 \end{equation}
\begin{lem}\label{lem:Z_close_to_a}
There exists $N \in \N$ such that the following property holds: for all $(\gamma,\phi,T) \in \V_\delta$ there exists $X\in \frac{1}{N}\Lambda_\z$ such that 
$$|X-p_\z \ol \phi| < \sqrt{\delta}.$$
\end{lem}
\begin{proof}
Define the path  $\mu:[0,1] \to Q, \ \mu(t) := \rho(t)\gamma(t)$, where
$$\rho:[0,1]\to G,\quad \rho(t)=\exp\left (\int_0^t \phi(s)\, ds\right )\,, \quad \forall\ t \in [0,1].$$ 
By \eqref{eq:Adidentity} we have
\[
\int_0^1 |\dot \mu|^2 \, dt  =\int_0^1 |\dot \gamma + \ul \phi(\gamma)|\, dt <\delta\,.
\]
This shows in particular that
\[
\mathrm{dist}\big (\gamma(0),\exp(\ol \phi)\gamma(0)\big ) = \mathrm{dist}\big (\mu(0),\mu(1)\big ) <\sqrt \delta.
\]
If $\Gamma_{\gamma(0)} \subset G$ denotes the stabilizer at $\gamma(0)$, then the previous inequality yields
\[
\mathrm{dist}\big(g,h) <\sqrt{\delta},\qquad h= \exp(\ol \phi)\,,
\]
for some $g \in \Gamma_{\gamma(0)}$.
Since $Q$ is compact there exists -- up to conjugation -- only finitely many different subgroups which appear as stabilizer groups. Let $N$ be the product of their orders. It follows that $g^N=e$. Using the triangle inequality combined with the invariance of the distance on $G$ with respect to right and left multiplication we then obtain
\begin{equation}\label{eq:hNclose}
\dist(e,h^N)<N\sqrt{\delta},\qquad h^N = \exp(N\ol \phi)\,.
\end{equation}
Now let $T \subset G$ be a maximal torus with Lie algebra $\t$ and integer lattice $\Lambda=\{X \in \t \mid \exp(X)=e\}$. By \cite[Thm.\ 6.4]{Bredon:Trafo} there exists $\ell \in G$ such that $\Ad_\ell \ol \phi \in \t$. Since $\exp:\t \to T$ is a local isometry we deduce from ~\eqref{eq:hNclose} that there exists $X \in \Lambda$ such that
\[
|\Ad_\ell N \ol \phi - X| < N \sqrt{\delta}\,.
\]
Moreover, since the splitting  $\g=\z\oplus \ker p_\z$ is left invariant by the adjoint action we have $p_\z \Ad_\ell \ol \phi = p_\z\ol\phi$ and hence
\[
|Np_\z \ol \phi-p_\z X| =|p_\z(N\Ad_\ell\ol \phi -X)| \leq |\Ad_\ell N\ol \phi-X| <N \sqrt{\delta}\,.
\]
The claim follows from the fact that $p_\z(\Lambda) = \Lambda_\z$. 
\end{proof}
\noindent By the previous lemma we see that for all $\delta>0$ small enough we have 
\[
\V_{\delta} = \bigsqcup_{X \in \frac{1}{N}\Lambda_\z} \V_{\delta,X},\qquad \V_{\delta,X}:= \Big \{(\gamma,\phi,T) \in \V_\delta \ \Big |\  |p_\z\ol \phi -X|<\sqrt \delta\Big \}.
\]
Observe that $\V_{\delta,X}$ might be empty. However if there exists $p \in Q$ such that $\exp(X) \in \Gamma_p$, where $\Gamma_p\subset G$ denotes the stabilizer subgroup at $p$, then $\V_{\delta,X}$ is non-empty and contains the vertical element $(\gamma_X,\phi_X,T)$ given by $\phi_X(t)=X$ and $\gamma_X(t)=\exp(tX)p$, for all $t \in [0,1]$. For such 
a vertical element we have 
\[\Sb_k (\gamma_X,\phi_X,T) = \<X,Z\> + kT.\]
\begin{lem}\label{lem:boundary}
For $\delta>0$ small enough there exists $\epsilon>0$ such that for all $X \in \frac{1}{N}\Lambda_\z$ with $\V_{\delta,X} \neq \emptyset$ we have
\[
\inf_{\V_{\delta,X}} \Sb_k = \<X,Z\>\,,\qquad \inf_{\partial \V_{\delta,X}} \Sb_k > \<X,Z\> +\epsilon\,.
\]
\end{lem}
\begin{proof}
For every $(\gamma,\phi,T) \in \partial\V_{\delta,X}$ we compute
\[
\Sb_k(\gamma,\phi,T) =\frac{\delta}{2T} +\<\ol \phi,Z\> + kT\geq \sqrt{2 \delta k} + \<\ol \phi,Z\> \geq \sqrt{2\delta k} -\sqrt{\delta} +\<X,Z\>\,,
\]
where the first inequality is obtained by minimizing in $T$ and the last is given by Lemma~\ref{lem:Z_close_to_a}. Notice indeed that by construction $Z\in \z$ and 
hence 
$$\<\ol \phi,Z\>=\<p_\z \ol \phi, Z\>.$$ 
The second assertion follows, as $(\sqrt{2k}-1)\sqrt{\delta}>0$. The first assertion follows as well, since any $(\gamma,\phi,T) \in \mathcal{V}_{\delta,X}$ is contained $\partial \mathcal{V}_{\delta',X}$ for $\delta':=\int_0^1 |\dot \gamma+\ul\phi(\gamma)|^2.$
\end{proof}
 
\begin{cor}\label{cor:palaissmale}Let $(\gamma_h,\phi_h,T_h)$ be a Palais-Smale sequence for $\Sb_k$ such that $T_h\to 0$. Then, after gauging and 
 taking a subsequence if necessary, we have that 
 $$\Sb_k(\gamma_h,\phi_h,T_h) \to \<X,Z\>,$$ 
 for some $X \in \frac{1}{N}\Lambda_\z$ and, for any $\delta >0$, 
the sequence $(\gamma_h,\phi_h,T_h)$ eventually enters the set $\V_{\delta,X}$.
\end{cor}
\begin{proof}
Fix $\delta>0$. By the first equation of~\eqref{2} we see that
\[
\int_0^1 |\dot \gamma_h + \ul{ \phi_h}(\gamma_h)|^2\, dt  = 2T_h^2(k+o(1))=2T^2_hk +o(T_h^2)\,.
\]
In particular $(\gamma_h,T_h,Z_h) \in \mathcal V_\delta$  for $h$ large enough. By Step 1 in the proof of Lemma \ref{lem:phibded} we can find an admissible 
sequence $(\rho_h)_h\subseteq \G$ such that the gauged sequence $(\rho_h\cdot \phi_h)$ is uniformly bounded in $L^2$. We pick one such admissible sequence 
and consider the gauged sequence $(\rho_h \cdot (\gamma_h,\phi_h,T_h))$. For sake of simplicity we will denote the gauged sequence again with $(\gamma_h,\phi_h,T_h)$.

By Lemma \ref{lem:Z_close_to_a}, for every $h\in \N$ there exists $X_h \in \frac{1}{N}\Lambda_\z$ such that  
\begin{equation}\label{eq:phiX}
|p_\z\ol \phi_h -X_h| < \sqrt{2(k+o(1))}T_h =\sqrt{2k}T_h +o(T_h)\,.
\end{equation}
It follows that
\begin{align*}
  \Sb_k(\gamma_h,\phi_h,T_h) &=\frac{1}{2T_h}\int_0^1 |\dot \gamma_h + \ul{\phi_h}(\gamma_h)|^2\, dt + \<\ol \phi_h,Z\> + kT_h\\
  & \leq 2kT_h+\<X_h,Z\> +\sqrt{2k}T_h + o(T_h)\\
  & < \<X_h,Z\> +o(1)\,,
\end{align*}
for $h$ large enough. On the other hand
\[\Sb_k(\gamma_h,\phi_h,T_h)\geq 2kT_h + \<X_h,Z\> - \sqrt{2k}T_h + o(T_h)\geq \<X_h,Z\>\,,\]
for $h$ large enough, as $k>1/2$. Since $\phi_h$ is uniformly bounded in $L^2$, $|\ol \phi_h|$ is uniformly bounded as well. Therefore, the set 
$$\Big \{\<X_h,Z\>\ \Big |\ h\in \N\Big \}\subseteq \R$$
is discrete (actually finite). We conclude that there exists $X \in \frac{1}{N}\Lambda_\z$ such that $X_h=X$ for every
$h$ large enough. In particular, $\Sb_k(\gamma_h,\phi_h,T_h) \to \<X,Z\>$ 
and, in virtue of \eqref{eq:phiX}, $(\gamma_h,\phi_h,T_h) \in \mathcal{V}_{\delta,X}$ for $h$ large enough, as we wished to prove. 
\end{proof}

The next lemma shows that maximal flow-lines of $\mathcal{X}_k$ that are defined on a finite interval have to approach vertical elements. 
The proof is analogous to the one of \cite[Lemma 4.9]{Asselle:2017} and will be omitted. 

\begin{lem}\label{noncompleteness}
Suppose $u:[0,R)\to \mathcal M$ is a maximal flow-line of $\mathcal{X}_k$, then there exist $X\in \frac{1}{N}\Lambda_\z$ and a sequence
$r_h\uparrow R$ such that $u(r_h) \in \mathcal{V}_{\delta,X}$ for all $h$ large enough and with $(\gamma_h,\phi_h,T_h):= u(r_h)$ we have
\[\int_0^1 |\dot{\gamma}_h+\ul{\phi_h}(\gamma_h)|^2\, dt \to 0, \quad \Sb_k(u(r_h))\to \<X,Z\>\,.\]
\end{lem}

Using Lemma \ref{noncompleteness} it is now easy to get from $\Phi^k$ a complete flow by stopping flow-lines which enter the sets  
\[\left\{\Sb_k<\<X,Z\> +\epsilon\right\} \cap \mathcal V_{\delta,X}, \quad \forall X \in \frac{1}{N}\Lambda_\z.\]
With slight abuse of notation, we denote the complete flow also with $\Phi^k$.

%%%%%%%%%%%%%%%%%%%%%%%%%%%%%%%%
%%%%%%%%%%%%%%%%%%%%%%%%%%%%%%%%
%%%%%%%%%%%%%%%%%%%%%%%%%%%%%%%%

\section{Proof of Theorems \ref{thm:main} and~\ref{thm:main3}}
\label{s:proofoftheorem1.1}

%%%%%%%%%%%%%%%%%%%%%%%%%%%%%%%%

In this section, building on the results of the previous ones, we
prove Theorem \ref{thm:main} and~\ref{thm:main3}. In order to show the existence of
critical points of $\Sb_k$, we will use the topological assumption on
$M$ to find a suitable (non-trivial) minimax class on the Hilbert
manifold $\mathcal M$ and a corresponding minimax function. We will
then show that such a minimax function yields critical points of
$\Sb_k$ for almost every $k>\frac 12$.

 The proof for Theorem~\ref{thm:main} follows closely \cite{Asselle:2017}; however, some extra care is needed in the whole line of argument. Indeed, on the one hand the functional $\Sb_k$ satisfies the Palais-Smale condition on $\M_{[T_*,T^*]}$ only up to gauge transformations and, 
on the other hand, the construction of the minimax class requires techniques coming from rational homotopy theory and orbifold theoretical homotopy theory.

%%%%%%%%%%%%%%%%%%%%%%%%%%%%%

\subsection{The minimax class for Theorem \ref{thm:main}}
\label{s:1} We call $M$ \emph{rationally
  aspherical} if $\pi^\orb_\ell(M)$ $\otimes \Q$ is trivial for all $\ell \geq 2$, where by $\pi^\orb_\ell(M)$ we denote the orbifold-theoretic homotopy group as defined in \cite[Def.\ 1.50]{Orbifolds}. Recall that with notation from Proposition~\ref{prop:orb} and \cite[Prop.\ 1.51]{Orbifolds} the orbifold homotopy group of $M$ are the (classical) homotopy groups of the Borel quotient $BM:=Q\times_G EG$, where $EG$ denotes the universal $G$-bundle. For every $k\geq 1$ we obtain an homomorphism 
\begin{equation}\label{eq:tauk}
\tau_k:\pi_k(Q) \to \pi_k^\orb(M)\,,
\end{equation}  which is induced by the quotient map  $Q \times EG \to BM$.    
\begin{lem}\label{lem:enonaspherical}
Assume that $\pi_k^\orb(M) \otimes \Q \neq 0$ for some $k \geq 2$, then there exists a class $a \in \pi_\ell(Q)$ for some $\ell \geq 2$, such that $\tau_\ell(a)$ has infinite order.
\end{lem}
\begin{proof}
The fibration $G \hookrightarrow Q \times EG \to BM$ induces an exact homotopy sequence
\[
\dots \stackrel{\tau_{k+1}}{\longrightarrow} 
\pi_{k+1}^\orb(M) \longrightarrow \pi_k(G) \longrightarrow \pi_k(Q) \stackrel{\tau_k}{\longrightarrow} \pi_k^\orb(M)\longrightarrow \dots\,.
\]
If $\tau_k\otimes \Q$ is trivial for all $k\geq 2$, then the sequence splits into short exact sequences
\[
0 \to \pi_{k+1}^\orb(M)\otimes \Q \to \pi_k(G)\otimes \Q \to \pi_k(Q)\otimes \Q\to 0\,,\quad \forall\ k \geq 1\,.
\]
In particular we have that
\[
\dim \pi_{k+1}^\orb(M) \otimes \Q \leq \dim \pi_k(G) \otimes \Q\,, \quad  \forall\ k \geq 1\,.
\]
Let $\widetilde{BM}$ be the (classical) universal cover of $BM$. Since $\pi_k^\orb(M)=\pi_k(BM) \cong \pi_k(\widetilde{BM})$ for all $k \geq 2$, we obtain 
\[
\dim \pi_{k+1}(\widetilde{BM}) \otimes \Q \leq \dim \pi_k(G) \otimes \Q\,,\quad   \forall\ k \geq 1\,.
\]
Furthermore, since $\pi_{2j}(G)\otimes \Q$ is trivial for  all $j \geq 1$  (cf.\ \cite[\S 15(f)]{FelixHalperinThomas}) we see that  $\pi_{2j+1}(\widetilde{BM})\otimes \Q$ is trivial for all $j \geq 0$. By the same inequality we conclude that $\pi_*(\widetilde{BM})\otimes \Q$ has finite type and thus  $H_*(\widetilde{BM},\Q)$ as well, following the remark after \cite[Thm.\ 15.11]{FelixHalperinThomas}. By \cite[Thm.\ 15.11]{FelixHalperinThomas} we see that the minimal model $(V,d)$ of $\widetilde{BM}$ has only even generators, which implies that the differential $d$ is trivial. The same theorem implies that $V^k$ is non-trivial for some $k \geq 2$, since by assumption $\pi_k^\orb(M)\otimes \Q\cong \pi_k(\widetilde{BM})\otimes \Q \neq 0$. Let $x \in V^k$ be a non-trivial element. Since $V$ is a free algebra, its powers $x^j \in V^{jk}$ are non-trivial for all $j \geq 1$. By the property of a minimal model and the vanishing of the differential of $V$, we have that $V^* \cong H^*(\widetilde{BM},\Q)$ and in particular $H^{jk}(\widetilde{BM},\Q)$ is non-trivial for all $j\geq 1$. 

Now let $\widetilde{M}$ be the universal orbifold cover of $M$ in the sense of \cite[Def.\ 2.16]{Orbifolds}. By \cite[Prop.\ 2.17]{Orbifolds} the Borel quotient corresponding to $\widetilde{M}$ is $\widetilde{BM}$. The Vietoris-Begle theorem yields an isomorphism $H^*(\widetilde{BM},\Q) \cong H^*(\widetilde{M},\Q)$ where the right-hand denotes the singular cohomology of the underlying topological space (cf.\ \cite[Prop.\ 2.12]{Orbifolds}). It follows that $H^*(\widetilde{M},\Q)$ is non-trivial in arbitrary large degrees, which is impossible for the finite-dimensional orbifold $\widetilde{M}$. 
\end{proof}
For $\ell \in \N$ let $B^{\ell} \subset \R^{\ell}$ the standard ball with boundary $S^{\ell-1}$. We identify the space $Q$ with the subspace of constant loops in $W^{1,2}(S^1,Q)$. Further we set $Q_{T_0} := Q \times\{0\} \times (0,T_0] \subset \M$, for $T_0 >0$ fixed. By equation \eqref{eq:levelatvertical} we have
\[
\max_{Q_{T_0}} \Sb_k = kT_0 \leq \e/2
\] 
if $T_0>0$ is chosen small enough, where $\e>0$ is the constant from Lemma~\ref{lem:boundary}. It is well-known that with any continuous map 
$u:(B^{\ell-1},S^{\ell-2})\to (\M,Q_{T_0})$
 we can associate a continuous map $v=v_u:S^\ell\to Q\times \g\times (0,\infty)$ and, conversely, with every smooth map $v:S^\ell\to Q \times \g \times (0,T_0]$ we can
 associate a continuous map of pairs of spaces $u=u_v:(B^{\ell-1},S^{\ell-2})\to (\M,Q_{T_0})$
such that $v_{u_v}$ is homotopic to $v$. Moreover a homotopy of $u$ induces a homotopy of $u_v$ and vice versa (cf.\ \cite[Proof of Theorem 2.4.20]{Klingenberg95} for more details). By abuse of notation we denote by $[u] \in \pi_\ell(Q)$ the homotopy class associated to $v_u$, where we have additionally identified $\pi_\ell(Q \times \g \times(0,\infty)) \cong \pi_\ell(Q)$ canonically. 

\begin{lem}\label{lem:atorsion}
There exists $\delta>0$ such that for any $u:(B^{\ell-1},S^{\ell-2}) \to (\M,Q_{T_0})$ satisfying $u(x) \in \V_{\delta}$ for all $x\in B^{\ell-1}$ we have that $\tau_\ell([u])=0$ in $ \pi_\ell^\orb(M)$.
\end{lem}
\begin{proof}
Let $T \subset G$ be a maximal tours with Lie algebra $\t$ and unit lattice $\Lambda \subset \t$. For $x \in B^{\ell-1}$ write $u(x)=(\gamma_x,\phi_x,T_x)$ and let $\ol \phi_x \in \g$
be as defined in \eqref{olphi}. As explained in the proof of Lemma~\ref{lem:Z_close_to_a} for any $x \in B^{\ell-1}$ we find $g=g_x\in G$ and $X =X_x \in \frac{1}{N}\Lambda$ such that $|\Ad_g \ol \phi_x -X| <\sqrt \delta$. Hence, either $|\ol\phi_x| <\sqrt\delta,$
or
$$|\ol \phi_x| > \Delta -\sqrt \delta,\quad \Delta:=\min_{X \in \frac{1}{N}\Lambda \setminus \{0\}} |X|>0.$$

For $\delta$ small enough the statements are mutually exclusive and since $x \mapsto |\ol \phi_x|$ is continuous and vanishes for all $x \in S^{\ell-2}$, we conclude that $|\ol\phi_x|<\sqrt \delta$ for all $x\in B^{\ell-1}$. Now set $(\gamma_x^{\rho_x},\phi_x^{\rho_x},T_x):=\rho_x\cdot(\gamma_x,\phi_x,T_x)$, where
\[
\rho_x:S^1 \to G,\qquad t \mapsto \exp\left(\int_0^t \phi_x(t')\, dt' - \ol\phi_x t'\right)\,.
\]
As in the proof of Lemma~\ref{lem:phibded} we see that 
$|\phi^{\rho_x}_x| \leq |\ol\phi_x| <\sqrt \delta$ 
and hence, using~\eqref{eq:dotgammasmall},
\[
\int_0^1 |\dot \gamma_x^{\rho_x}|^2\, dt< 4 \delta.
\]
In particular, for $\delta>0$ small enough the map $x \mapsto \gamma_x^{\rho_x}$ defines the trivial homotopy class in $\pi_\ell(Q)$ (cf. \cite[Thm.\ 1.4.15]{Klingenberg:lectures}). We conclude that $[u]$ is in the image of $\pi_\ell(G) \to \pi_\ell(Q)$ and in particular mapped to the trivial class to $\pi_\ell^\orb(M)$. 
\end{proof}

Given the homotopy class $a \in \pi_\ell(Q)$ as in Lemma~\ref{lem:enonaspherical}, we now define
$$\P:= \Big\{u:(B^{\ell-1},S^{\ell-2})\to (\M,Q_{T_0}) \Big |\  [u]=a\Big\}.$$
We readily see that $\P\neq \emptyset$, since $u_v \in \P$ for any $v:S^\ell\to Q \times \{0\} \times (0,T_0]$ smooth such that $[v]=a$. Obviously, $\P$ 
is invariant under the complete flow defined in Subsection~\ref{sec:completeflow}, provided that $T_0>0$ is small enough. The last property of $\P$
we need is that every element $u\in \P$ has to intersect $\partial \V_\delta$ (more precisely, $\partial \V_{\delta,0}$). 
This follows trivially from Lemma~\ref{lem:atorsion}. 

%%%%%%%%%%%%%%%%%

\subsection{The minimax class for Theorem \ref{thm:main3}}
\label{s:2}
We adapt the argument in \cite{Haefliger:2006} to our setting. Consider a tube $G \times_\Gamma U$, where $\Gamma\subset G$ is a  stabilizer group and $U \subset Q$ is a contractible slice (i.e. a submanifold which is $\Gamma$ invariant). From the $G$-equivariant embedding $G \times_\Gamma  U \hookrightarrow Q$ we obtain an embedding 
\[ U\times_\Gamma EG \cong (G \times_\Gamma U)\times_G EG \hookrightarrow Q \times_G EG=BM\,,\] which induces a group homomorphism 
\begin{equation}\label{eq:homo}
\rho:\Gamma \cong \pi_1(U \times_\Gamma EG) \to \pi_1(BM)\cong \pi_1^\orb(M)\,.
\end{equation}
Such a homomorphism is precisely the homomorphism defined in~\cite[Lemma 2.22]{Orbifolds}. We deduce that, if $M$ is not developable, we can find a tube $U\times_\Gamma EG$ and a non-trivial element $[\bar \gamma]$ in $\pi_1^\orb(U)$ which is trivial in $\pi_1^\orb(M)$. Consider the diagram with exact rows
\[
\xymatrix{
\pi_1(G)\ar[r]\ar[d]^{=}&\pi_1(G \times_\Gamma U)\ar[d]\ar[r]&\pi_1^\orb(U)\ar[d]\ar[r]&\pi_0(G)\ar[d]^{=}\\
\pi_1(G)\ar[r]&\pi_1(Q)\ar[r]&\pi_1^\orb(M)\ar[r]
&\pi_0(G)}
\]
A simple diagram chase shows that there exists also a class $[\gamma]$ in $\pi_1(G\times_\Gamma U)$ 
which is trivial in $\pi_1(Q)$. Without loss of generality we assume that the representative $\gamma$ is vertical and $X=\theta(\dot \gamma)$ constant in $t$. Using the homotopy of $\gamma$ to a constant loop in $Q$ we now define a non-trivial minimax class for the functional $\Sb_k$. More precisely, consider the space of continuous maps
\[
\P := \Big \{u:[0,1] \to \M \ \Big |\  u(0) = (\gamma,-X,T_0),\ u(1) \in Q_{T_0} \Big \}.
\]
Clearly, $\mathcal P$ is non-empty and invariant under the flow $\Phi^k$ defined in Section \ref{sec:completeflow}, provided that $T_0>0$ is chosen small enough. The last thing we need to check is that every $u\in \P$ has to intersect $\partial \V_\delta$ for all $\delta>0$ sufficiently small. For any $x\in[0,1]$ we write $u(x)=(\gamma_x,\phi_x,T_x)$, set $\ol \phi_x$ as in \eqref{olphi}, and assume by contradiction that $u(x) \in \V_\delta$ for all $x\in[0,1]$. Exactly as in the proof of Lemma~\ref{lem:atorsion} we see that for all $x\in[0,1]$ we have either $|\ol \phi_x| <\sqrt \delta$ or $|\ol \phi_x| > \Delta - \sqrt{\delta}.$
Since $|\ol \phi_0| = |X| > \Delta -\sqrt{\delta}$, $|\ol \phi_1| =0$, and the two conditions above  are
mutually exclusive if $\delta$ is small enough, we obtain a contraction to the continuity of the map  $x\mapsto |\ol\phi_x|$.

%%%%%%%%%%%%%%%%%

\subsection{End of proofs} We define the minimax function 
$$c:(1/2,+\infty)\to (0,+\infty),\qquad c(k):= \inf_{u\in\P}\max \Sb_k\circ u,$$
where $\mathcal P$ is the minimax class defined in Section \ref{s:1} resp. \ref{s:2}.
By Lemma~\ref{lem:boundary} we have $c(k)\geq \epsilon$ for all $k > 1/2$, for every $u\in\mathcal P$ has to intersect $\partial \V_{\delta}$. 
However, this is not enough to exclude that $T_h$  converges to zero as $h\to +\infty$ for some Palais-Smale sequence $(\gamma_h,\Phi_h,T_h)$ for $\Sb_k$ at level $c(k)$, 
as it might be that $c(k)=\<X,Z\>$ for some $X \in \frac{1}{N}\Lambda_\z$.  For that we will need the piece of additional
information given by the following lemma. For the proof we refer to \cite[Lemma 5.3]{Asselle:2017}.

\begin{lem} Let $u$ be any element of $\P$. Suppose that $x^*\in B^{\ell-1}$ (resp. $[0,1]$) is such that
\begin{equation}
\Sb_k(u(x^*))\geq \max \Sb_k\circ u - \epsilon/2.
\label{almostmaximum}
\end{equation}
Then $u(x^*)\notin \bigcup_{X\in \frac{1}{N}\Lambda_\z} \{\Sb_k< \<X,Z\> + \epsilon/2\}\cap \mathcal V_{\delta,X}$.
\label{lem:almostmaximum}
\end{lem}
%\begin{proof}
%Suppose by contradiction that there exists $X \in \frac{1}{N}\Lambda_\z$ such that 
%$$u(x^*)\in \big \{\Sb_k< \<X,Z\> + \epsilon/2\big \}\cap \mathcal V_{\delta,X}.$$ 
%By Lemma~\ref{lem:atorsion} there 
%exists $x\in B^{\ell-1}$ such that $\vp(x)\in \partial \V_{\delta,X}$. Using Lemma \ref{lem:boundary} we now readily see that 
%\begin{align*}
%  \max_{B^{\ell-1}} \Sb_k \circ u - \Sb_k(u(x^*))&\geq \Sb_k(u(x)) - \Sb_k(u(x^*))> \epsilon/2,
%\end{align*}
%in contradiction with \eqref{almostmaximum}.
%\end{proof}

The function $c(\cdot)$ is monotonically increasing in $k$
and hence almost everywhere differentiable. 
 The next proposition shows there exist Palais-Smale sequences
$(\gamma_h,\phi_h,T_h)\subseteq \mathcal M$ for $\Sb_k$ with with $T_h$'s
bounded away from zero and uniformly bounded, provided $k$ is a
point of differentiability for $c(\cdot)$.

\begin{prop}\label{prop:struwe}
  Let $k^*$ be a point of differentiability for $c(\cdot)$. Then there
  exists a Palais-Smale sequence $(\gamma_h,\phi_h,T_h)\subseteq
  \mathcal M$ for $\Sb_{k^*}$ such that the $T_h$'s are uniformly bounded and bounded away from zero.
\end{prop}

The proof relies on the celebrated \textit{Struwe monotonicity argument} \cite{Struwe:1990sd} and will be omitted since it is a plain adaptation of the proof of \cite[Proposition 5.4]{Asselle:2017} (see also \cite[Lemma 8.1]{Abbondandolo:2013is}  and 
\cite[Proposition 7.1]{Contreras:2006yo}) and \cite[Proposition
4.1]{Asselle:2014hc}), taking into account the fact that $\Sb_k$ satisfies the Palais-Smale condition on $\M_{[T_*,T^*]}$ only up to gauge transformations.

\begin{proof}[Proof of Theorems \ref{thm:main} and \ref{thm:main3}] Combine Proposition  \ref{prop:struwe} with Lemma \ref{lem:palaissmale}. 
\end{proof}

\begin{proof}[Proof of Corollary \ref{cor:main}]
We prove i). Up to passing to the compact orbifold universal cover we can assume that $\pi_1^{\text{orb}}(M)=0$. Notice that in this case $M$ must be orientable. Clearly, it suffices to show that $M$ is not rationally aspherical. Suppose by contradiction that $\pi_*(BM) \otimes \Q$ is trivial for all $*\geq 2$, where $BM$ is the classifying space of $M$. By assumption $\pi_1(BM)= 0$. In particular $BM$ is simply connected and its minimal model is trivial (cf.\ \cite[Thm.\ 15.11]{FelixHalperinThomas}). This implies that $H^*(BM;\Q)\cong 0$ for all $*\geq 1$. By~\cite[Prop.\ 2.12]{Orbifolds} we have $H^n(M;\Q) \cong H^n(BM;\Q)\cong 0$ with $n=\dim M$, which is impossible by Poincar\'e duality.

Now we prove ii). Assume by contradiction that the universal orbicover $\widetilde M$ is rationally aspherical. As in the proof of Lemma~\ref{lem:enonaspherical} we see it must have a minimal model $(V^*,d)$ with vanishing differential. Further by assumption $V^2\cong H^2(\widetilde M,\Q)$ is non-trivial, because the pull-back of $\sigma$ to $\widetilde M$ yields a non-trivial cohomology class. This shows that $V^*\cong H^*(\widetilde M,\Q)$ is infinite dimensional in contradiction to the assumption that $M$ is a finite dimensional orbifold. 
\end{proof}

%%%%%%%%%%%%%%%%%%%%%%%%%%%%%%%%%
%%%%%%%%%%%%%%%%%%%%%%%%%%%%%%%%%
%%%%%%%%%%%%%%%%%%%%%%%%%%%%%%%%%

\section{Proof of Theorem~\ref{thm:main2}}
\label{s:proofoftheorem1.5}

Let $(M,g_M)$ be a closed Riemannian orbifold. By Proposition~\ref{prop:quotient_orbifold} there exists a smooth manifold $Q$ equipped with a locally free action of a compact group $G$ such that $Q/G \cong M$ as orbifolds. Without loss of generality we assume that $G$ is connected. Indeed, if $G_0 \subset G$ denotes the connected component of the neutral element, the quotient $Q/G_0$ is a finite cover of $Q/G$ and a closed geodesic in the former yields one in the latter. 
%Also, by Corollary~\ref{cor:singular_set_finite} we can assume without loss of generality that the singular set of $M$ is finite. 
Now consider the metric $g_Q$ on $Q$ associated with $g$ as explained in Section 2 and define the functional 
\[\E :W^{1,2}(S^1,Q)\times L^2(S^1,\g)\to \R,\quad \E (\gamma,\phi):= \int_0^1 |\dot \gamma + \ul \phi (\gamma)|^2 \, dt\,,\]
where $\g$ denotes the Lie algebra of $G$ and $|\cdot|$ the norm on $TQ$ induced by $g_Q$. From the discussion in Section 5 it follows immediately  
that the functional $\E$ is smooth and bounded from below (by zero). Moreover, it is invariant under gauge transformations, satisfies the Palais-Smale condition up to gauge transformations, and its critical points project to closed geodesic in $(M,g_M)$. In particular, critical points contained in  $\E^{-1}(0)$ project to point curves.

\begin{proof} [Proof of Theorem \ref{thm:main2}] 
If $M$ is not developable, then we obtain a non-constant closed geodesic on $M$ by literally repeating the proof of Theorem \ref{thm:main3}. 
Indeed, consider a non-zero class $[\gamma]$ in $\pi_1(G\times_\Gamma U)$ 
which is trivial in $\pi_1(Q)$. We assume that the representative $\gamma$ is vertical and $X=\theta(\dot \gamma)$ constant in $t$ and define the following minimax class for the functional $\E$:
\[
\P_0 := \Big \{u:[0,1] \to W^{1,2}(S^1,Q)\times L^2(S^1,\g) \ \Big |\  u(0) = (\gamma,-X),\ u(1) = (q_0,0)\Big \}.
\]
Clearly, $\mathcal P_0$ is non-empty and invariant under the negative gradient flow of $\E$. Moreover, every $u\in \P_0$ has to intersect $\partial \V_\delta$ for all $\delta>0$ sufficiently small, where 
$$\mathcal V_\delta := \Big \{(\gamma,\phi)\in W^{1,2}(S^1,Q)\times L^2(S^1,\g)\ \Big |\ \E(\gamma,\phi)<\delta\Big \}.$$ 
Therefore, the minimax value 
$$c:= \inf_{u\in \P_0 }\max_{x\in [0,1]} \E(u(x))$$
is strictly larger than $\delta$. This yields the existence of a critical point for $\E$ at level $c$, thus of a non-constant closed geodesic in $M$, as the functional $\E$ satisfies the Palais-Smale condition up to gauge transformations.

If $\pi_1^{\text{orb}}(M)$ is finite, then up to passing to the orbifold universal cover, we see that $M$ is not rationally aspherical and hence the proof follows by the same argument used for Theorem \ref{thm:main}.
Indeed, consider $a\in \pi_\ell(Q)\otimes \Q$ such that $\tau_\ell(a)\not = 0 \in \pi_\ell^{\text{orb}}(M)\otimes \Q$ for some $\ell \geq 2$. Denote by $\P$ the space of continuous maps 
\[u:(B^{\ell-1},S^{\ell -2}) \to (W^{1,2}(S^1,Q)\times L^2(S^1,\g), Q\times \{0\})\,,\]
representing the homotopy class $a$. Since every $u \in \P$ has to intersect the boundary of the set $\mathcal V_\delta$,
for $\delta>0$ sufficiently small,  the existence of the desired non-constant closed geodesic follows by minimax over the class $\P$.

Finally, assume that $\pi_1^\orb(M)$ contains an element of infinite order, say $\tilde a$. Consider the exact homotopy sequence 
\begin{equation}\label{eq:homotopy_sequence}
\dots \to \pi_1(G) \to \pi_1(Q)\to \pi_1^{\text{orb}}(M)\to \pi_0(G)=0.
\end{equation}
From the surjectivity of the map $\tau_1:\pi_1(Q)\to \pi_1^{\text{orb}}(M)$ we deduce that there is an element $a\in \pi_1(Q)$ such that $\tau_1(a)=\tilde a$. In particular, 
$a$ has infinite order. 
Now consider the connected component $\mathcal C_a$ of $W^{1,2}(S^1,Q)\times L^2(S^1,\g)$ associated to $a$. We claim that there exists 
$\delta>0$ sufficiently small such that
$$\E(\gamma,\phi)\geq \delta, \quad \forall (\gamma,\phi)\in \mathcal C_a.$$
Indeed, assume by contradiction that we can find a sequence $(\gamma_h,\phi_h) \subset \mathcal{C}_a$ such that $\E(\gamma_h,\phi_h)\to 0$; then, $\gamma=\gamma_h$ lies in some tube $G\times_\Gamma U$, with $U$ contractible, for some $h$ large enough.
But this show that up to conjugation $a$ lies in the image of $\pi_1(G \times_\Gamma U) \to \pi_1(Q)$. This shows in particular that, up to conjugation, $\tilde a$ lies in the image of $\Gamma \cong \pi_1^\orb(U) \to \pi_1^\orb(M)$, hence must have finite order in contradiction with our assumption. The existence of the required closed geodesic in $M$ follows now by minimizing the functional $\E$ over the connected component $\mathcal C_a$. 
\end{proof}

\appendix

\section{A complete proof of Lemma \ref{lem:palaissmale}} 
\label{App:functionalsk}

%Here we prove those properties of the functional $\Sb_k$ which have not been proved in the main body of this paper in order to keep the exposition as reader friendly as possible. 
%The first fact we need to check is the regularity of $\Sb_k$.
%
%\begin{lem}\label{lem:lipschitz}
%The functional $\Sb_k$ is smooth.
% belongs to $C^{1,1 (\mathcal M)$.
%\end{lem}
%\begin{proof}
%Fix an orthonormal basis $X_1,X_2,\dots,X_{\dim G}$ of $\g$ and rewrite the functional
%\begin{multline*}
%\Sb_k(\gamma,\phi,T) = \frac{1}{2T}\int_0^1|\dot \gamma|^2\,dt 
%+ \frac 1{2T}\int_0^1 |\phi|^2\,dt + \int_0^1\<\phi,Z\>\,dt + kT+\\
%+\frac 1T \int_0^1 \<\phi,X_1\> \<\dot \gamma,\ul X_1(\gamma)\>\, dt +\dots +\frac{1}{T} \int_0^1\<\phi,X_{\dim G}\>\<\dot \gamma,\ul X_{\dim G}\>\,dt
%\,.
%\end{multline*}
%Denote by $\Sb_k^j$ 
%the functionals obtained by mapping $(\gamma,\phi,T)$ to $j$-th term on the right-hand side. It suffices to show that $\Sb_k^j$ is smooth for all $j$. That $\Sb_k^1$ is smooth is standard (cf. \cite[Section 2.3]{Klingenberg95}). That $\Sb_k^j$ for $j=2,3,4$ is smooth is obvious. To show smoothness of $\Sb_k^j$ for $j\geq 5$ it remains to see that the maps $W^{1,2}(S^1,Q)\to L^2(S^1,\R)$, $\gamma \mapsto \<\dot \gamma,\ul X_i(\gamma)\>$ for all $i=1,\dots,\dim G$ are smooth. This follows again from standard theory and the fact that the fundamental vector field $\ul X_i$ is smooth.
%\end{proof}

\begin{lem}
Suppose $(\gamma_h,\phi_h,T_h)$ is a Palais-Smale sequence for $\Sb_k$ with $(\phi_h)$ uniformly bounded in $W^{1,2}$ and strongly converging (in $L^2$) to $\phi$, $T_h\to T$, and that $\gamma_h$  converges uniformly to some 
$\gamma\in C^0(S^1,Q)$. Then $\gamma \in W^{1,2}(S^1,Q)$ and $\gamma_h\to \gamma$ strongly in $W^{1,2}$.
\label{lem:palaissmale2}
\end{lem}

\begin{proof}
Let $\gamma_0$ be a smooth curve sufficiently close to $\gamma$. For all $h$ sufficiently large define sections in $\gamma_0^*\, TQ$ via 
$$\exp_{\gamma_0(t)}\zeta_h(t) := \gamma_h(t),\quad \forall t\in S^1.$$ 
Similarly define $\zeta$ via  $\exp_{\gamma_0(t)} \zeta(t) ;= \gamma(t),\ \forall t\in S^1.$
Since $\gamma_h$ converges uniformly to $\gamma$, the sections $\zeta_h$ converge uniformly to $\zeta$. Moreover since $\|\dot \gamma_h\|_2$ is uniformly bounded, so is $\|\na_t \zeta_h\|_2$. Thus $\zeta_h$ is bounded in $W^{1,2}(S^1,\gamma_0^*\, TQ)$ and hence converges weakly in $W^{1,2}$ to $\zeta$. This shows first of all that $\zeta \in W^{1,2}$ and thus $\gamma \in W^{1,2}$. Thus, it remains to show that $\zeta_h \to \zeta$ in $W^{1,2}$. Since we already know that $\zeta_h \to \zeta$ uniformly it suffices to show that $\na_t \zeta \to \na_t \zeta$ strongly in $L^2$. We compute
\begin{equation}
\label{eq:W12}
\|\na_t (\zeta_h-\zeta)\|^2_2\leq \<\na_t \zeta_h,\na_t(\zeta_h-\zeta)\>_2 +\<\na_t \zeta,\na_t(\zeta_h-\zeta)\>_2\,. 
\end{equation}
The last term is infinitesimal because $\na_t\zeta_h$ weakly converges to $\na_t\zeta$ in $L^2$. For each $t \in S^1$ let $\Pi_h(t):T_{\gamma_0(t)}Q \to T_{\gamma_h(t)}Q$ be the parallel transport along the shortest curve from $\gamma_0(t)$ to $\gamma_h(t)$, i.e. the curve 
$$[0,1] \ni s\mapsto \exp_{\gamma_0(t)}(s\zeta_h(t)).$$ 
Consider $\xi_h(t):=\Pi_h(t)(\zeta_h(t)-\zeta(t))$. One shows that $\xi_h \in W^{1,2}(\gamma_h^*\, TQ)$ and 
\[
\<\na_t \zeta_h,\na_t(\zeta_h-\zeta)\> = O(1)\<\dot \gamma_h,\na_t \xi_h\>_2 +O(1)\<\dot \gamma_0,\na_t(\zeta_h-\zeta)\>_2\,.\]
Again, the last term is infinitesimal because $\na_t \zeta_h$ weakly converges to $\na_t \zeta$ in $L^2$. Using \eqref{eq:W12} we obtain
\begin{equation}
\label{eq:W12cont}
\|\na_t (\zeta_h-\zeta)\|^2_2=  O(1)\<\dot\gamma_h,\na_t \xi_h\>_2+o(1)\,. 
\end{equation}
Since the norm is preserves by parallel transport we have
\begin{equation}
\label{eq:xihinfinitesimal}
\|\xi_h\|_\infty=\|\zeta_h-\zeta\|_\infty =o(1)\,.
\end{equation}
Moreover, we have that 
$$\|\xi_h\|_{1,2}=O(1)\|\zeta_h-\zeta\|_{1,2}= O(1)$$ 
is uniformly bounded. Since $(\gamma_h,\phi_h,T_h)$ is a Palais-Smale sequence, for any uniformly bounded sequence of sections $(\xi_h)\subset W^{1,2}(\gamma_h^*\, TQ)$ we have that
\begin{align*}
o(1) &= d\Sb_k(\gamma_h,\phi_h,T_h)[\xi_h,0,0]=\frac{1}{2T_h}\int_0^1\<\dot \gamma_h+\ul{\phi_h}(\gamma_h),\na_t \xi_h + \na_{\xi_h}\ul {\phi_h}\> \, dt\,.
\end{align*}
Multiplying by $2T_h$ and applying the same steps as in the proof of Lemma~\ref{lem:critpoints} yields
\[
o(1)=\int_0^1\<\dot \gamma_h,\na_t \xi_h\> + \sigma_{\phi_h}(\dot \gamma_h,\xi_h\> + \<\ul{\pt \phi_h}(\gamma_h),\xi_h\>\, dt,
\]
from which we conclude, using uniform boundedness of $\|\phi_h\|_{2}$, $\|\pt\phi_h\|_2$ and $\|\dot \gamma_h\|_2$, that
\[
\int_0^1 \<\dot \gamma_h,\na_t\xi_h\>\, dt= O(1)\|\xi_h\|_\infty+o(1)\,.
\]
The claim follows combining the inequality above with \eqref{eq:W12cont} and \eqref{eq:xihinfinitesimal}.
\end{proof}

\bibliography{_biblio}
\bibliographystyle{plain}

\end{document}